\documentclass[11pt, twoside, leqno]{article}

\usepackage{amssymb}
\usepackage{amsmath}
\usepackage{amsthm}
\usepackage{color}
\usepackage{mathrsfs}

\usepackage{indentfirst}

\usepackage{txfonts}

\allowdisplaybreaks

\def\XXint#1#2#3{{\setbox0=\hbox{$#1{#2#3}{\int}$ }
\vcenter{\hbox{$#2#3$ }}\kern-.6\wd0}}

\def\({\left(}
\def \){ \right)}

 \def\dim{\operatorname{dim}}

\newtheorem{theorem}{Theorem}[section]
\newtheorem{lemma}[theorem]{Lemma}

\theoremstyle{definition}

\renewcommand{\appendix}{\par
   \setcounter{section}{0}%
   \setcounter{subsection}{0}%
   \setcounter{subsubsection}{0}%
   \gdef\thesection{\@Alph\c@section}%
   \gdef\thesubsection{\@Alph\c@section.\@arabic\c@subsection}%
   \gdef\theHsection{\@Alph\c@section.}%
   \gdef\theHsubsection{\@Alph\c@section.\@arabic\c@subsection}%
   \csname appendixmore\endcsname
 }

\numberwithin{equation}{section}

\begin{document}

\arraycolsep=1pt

\title{\bf\Large A Carleson problem for the Boussinesq operator
\footnotetext{\hspace{-0.35cm} 2010 {\it
Mathematics Subject Classification}. 42B25, 35Q20.
\endgraf {\it Key words and phrases.} Carleson problem, Boussinesq operator, pointwise convergence, Hausdorff dimension, Sobolev space.
}
}
\author{ Dan Li, Junfeng Li\footnote{Corresponding author} ~and Jie Xiao}
\date{}
\maketitle

\vspace{-0.7cm}

\begin{center}
\begin{minipage}{13cm}
{\small {\bf Abstract}\quad
In this paper, Theorems \ref{theorem 1.1}-\ref{theorem 1.10} show that the Boussinesq operator $\mathcal{B}_tf$ converges pointwise to its initial data $f\in H^s(\mathbb{R})$ as $t\to 0$ provided $s\geq\frac{1}{4}$ - more precisely - on the one hand, by constructing a counterexample in $\mathbb{R}$ we discover that the optimal convergence index $s_{c,1}=\frac14$; on the other hand, we find that the Hausdorff dimension of the disconvergence set for $\mathcal{B}_tf$ is
		\begin{align*}
		\alpha_{1,\mathcal{B}}(s)&=\begin{cases}
		1-2s&\ \ \text{as}\ \ \frac{1}{4}\leq s\leq\frac{1}{2};\\
		1 &\ \ \text{as}\ \ 0<s<\frac{1}{4}.
		\end{cases}\\
		\end{align*}
		Moreover, Theorem \ref{theorem 1.7} presents a higher dimensional lift of Theorems \ref{theorem 1.1}-\ref{theorem 1.10} under $f$ being radial.

}
\end{minipage}
\end{center}


\section{Introduction}\label{s1}

\subsection{Carleson problem for Schr\"odinger operator}\label{s11}
For $(n,s)\in\mathbb N\times\mathbb R$ and the Schwartz space $S(\mathbb R^n)$, let
$$H^s(\mathbb{R}^n):=\left\{f\in S(\mathbb R^n):\ \|f\|_{H^s(\mathbb{R}^n)}=\left(\int_{\mathbb{R}^n} \left(1+|\xi|^2\right)^s \left|\hat{f}(\xi)\right|^2\,\textrm{d}\xi\right)^{\frac{1}{2}}<\infty\right\}$$
be the $s$-order $L^2$ Sobolev space on $\mathbb R^n$. Formally, the Schr\"odinger operator acting on $f\in S(\mathbb R^n)$
\begin{align}\label{Schrodinger}
\mathcal{S}_tf(x):=(2\pi)^{-n}\int_{\mathbb{R}^n} e^{ix\cdot\xi}e^{it|\xi|^2}\hat{f}(\xi)\,\textrm{d}\xi
\end{align}
solves the initial data problem of free Schr\"odinger equation
\begin{equation}\label{SE}\begin{cases}
		i\partial_t u+\Delta_x u=0;\\
		u(x,0)=f(x).
		\end{cases}
\end{equation}		

Due to a fundamental interest in mathematical and theoretical physics, in \cite{C} Carleson proposed a problem to determine the optimal order $s_{c,n}$ such that
\begin{align}\label{ConvergenceS}
\lim_{t\rightarrow0}\mathcal{S}_tf(x)=f(x), \quad\text{a.e.}\ \  x\in\mathbb{R}^n
\end{align}
holds for all $f\in H^s(\mathbb{R}^n)$ with $s\geq s_{c,n}$.

There are enumerate literatures devoted to this problem (see \cite{B3,B1,B2,L,MYZ,MVV,S1,S3,TV,V2,V1} and the references therein) - in particular -  $s_{c,n}=\frac{n}{2(n+1)}.$ However, it is perhaps appropriate to  mention some important steps toward the formula on $s_{c,n}$.
\begin{itemize}
\item[$\rhd$]
 For $n=1$, Carleson \cite{C} proved \eqref{Schrodinger} converges to its initial data with $s\geq \frac{1}{4}$ and Dahlberg-Kenig \cite{DK} gave counterexample to show that this convergence cannot be true for $s<\frac14$.
\item[$\rhd$] For $n\geq 2$, Bourgain \cite{B2} and Sj\"olin \cite{S4} formulated independently  counterexamples for  $s<\frac{n}{2(n+1)}$. Recently, a positive result has been established under $s>\frac{n}{2(n+1)}$ by Du- Guth-Li \cite{DGL} for $n=2$ and Du-Zhang \cite{DZ} for $n\geq 3$.
\end{itemize}

Furthermore, in \cite{SS} Sj\"{o}gren-Sj\"{o}lin refined Carleson's problem to determine
the Hausdorff dimension of the disconvergence set:
\begin{align}\label{DimensionS}
\alpha_{n,\mathcal{S}}(s):=\sup_{f\in H^s(\mathbb{R}^n)}\dim _H\left\{x\in\mathbb{R}^n: \lim_{t\rightarrow0}\mathcal{S}_tf(x)\neq f(x)\right\}.
\end{align}

By Sobolev embedding, we easily get $$\alpha_{n,\mathcal{S}}(s)=0\ \ \forall\ \ s>\frac{n}{2},
$$
thereby being led to consider the case $s\leq\frac n2$.
\begin{itemize}

\item[$\rhd$]
Bourgain's counterexample in \cite{B2} implies  $$\alpha_{n,\mathcal{S}}(s)= n\ \ \forall\ \ s<\frac{n}{2(n+1)}.
$$
\item[$\rhd$] Luc\`{a}-Rogers in \cite{LR2} proved
$$\alpha_{n,\mathcal{S}}(s)= n\ \ \text{as}\ \ s=\frac{n}{2(n+1)}.
$$
\item[$\rhd$] For $\frac{n}{4}\leq s\leq\frac{n}{2}$, we can combine the results in \v{Z}ubrini\'{c} \cite{Z} and Barcel\'{o}-Bennett-Carbery-Rogers \cite{BBCR} to obtain $\alpha_{n,\mathcal{S}}(s)=n-2s.$

\item[$\rhd$] Notice that if
$$n=1\ \ \&\ \  ~\frac{n}{2(n+1)}=\frac{n}{4}=\frac14,$$
then $\alpha_{1,\mathcal{S}}(s)=1-2s$. And yet, for
$$n\geq 2\ \ \&\ \ \frac{n}{2(n+1)}<s< \frac n4,
$$
nobody knows the value of $\alpha_{n,\mathcal{S}}(s)$; see also \cite{DGLZ,DZ,LR3,LR1,LR2} for more information.
\end{itemize}

\subsection{Carleson problem for  Boussinesq operator}\label{s12}
As a nonlinear variant of \eqref{Schrodinger}, the Boussinesq operator acting on $f\in S(\mathbb{R}^n)$ is defined by
\begin{align}\label{Boussinesq}
 u(x,t)
:=\mathcal{B}_tf(x)
:=(2\pi)^{-n}\int_{\mathbb{R}^n} e^{ix\cdot\xi}e^{it|\xi|\sqrt{1+|\xi|^2}}\hat{f}(\xi)\,\textrm{d}\xi,
\end{align}
which  occurs in a large number of physical situations, which has motivated their study in  physics and mathematics. The name of this operator comes from  the  Boussinesq equation (cf.\cite{Bous})
$$u_{tt}-u_{xx}\pm u_{xxxx}=(u^2)_{xx}\ \ \forall\ \ (t,x)\in \mathbb R^2
$$
modelling the propagation of long waves on the surface of water with small amplitude. Our interest in this operator arises from the study of the Gross-Pitaevskii (G-P) equation
\begin{equation*}\label{GP}
i\partial_t \psi+\Delta\psi=(|\psi|^{2}-1)\psi\ \ \text{subject to}\ \  \psi: \mathbb{R}^{1+n}\rightarrow\mathbb{C}\ \ \&\ \
\lim_{|x|\rightarrow\infty}\psi=1.
\end{equation*}
 The above nonzero boundary condition arises naturally in physical contexts such as Bose-Einstein condensates, superfluids and nonlinear optics, or in the hydrodynamic interpretation of NLS (cf. \cite{FS}). There are many literatures studying the existence and asymptotic behaviour of a solution to the G-P equation. For the most recent progress on these topics we refer the readers to \cite{BGS, Ge, GNT1, GNT2, GNT3, GHN}. Even though \eqref{Boussinesq} is very close to \eqref{Schrodinger},  the constant boundary condition brings a remarkable effect on the space-time behaviour of a solution - this actually is one of the main motivations of this paper.

 Quite surprisingly, upon letting $v=\psi-1$, we find
$$i\partial_tv +\Delta v-2\Re v=v^2+2|v|^2+|v|^2v,$$
thereby using the diagonal transform
$$v=v_1+iv_2\rightarrow u=u_1+iu_2:=v_1+iUv_2,
$$
to get the following system for $(u,v):$
$$
\begin{cases} i\partial_t u-Hu=U(3v_1^2+v_2^2+|v|^2v_1)+i(2v_1v_2+|v|^2v_2);\\
U:=\sqrt{-\Delta(2-\Delta)^{-1}};\\
H:=\sqrt{-\Delta(2-\Delta)}.
\end{cases}
$$
Accordingly, \eqref{Boussinesq} solves the initial data problem of the induced Boussinesq equation
\begin{equation}\label{GP}\begin{cases}
		i\partial_t u-Hu=0;\\
		u(x,0)=f(x).
		\end{cases}
\end{equation}

In this paper, we are motivated by \S \ref{s11} and similarity between \eqref{Schrodinger} (solving \eqref{SE}) and \eqref{Boussinesq} (solving \eqref{GP}) to consider:
\begin{enumerate}
\item[I.] The Carleson problem for $\mathcal{B}_tf(x)$: evaluating the optimal $s_{c,n}$ such that
\begin{align}\label{Convergence}
\lim_{t\rightarrow0}\mathcal{B}_tf(x)=f(x), \quad\text{a.e.}\ \  x\in\mathbb{R}^n
\end{align}
holds for any $f\in H^s(\mathbb{R}^n)$  with $s\geq s_{c,n}$.
\item[II.] A refinement of the Carleson problem for $\mathcal{B}_tf(x)$: determining the Hausdorff dimension of the disconvergence set:
\begin{align}\label{Dimension}
\alpha_{n,\mathcal{B}}(s):=\sup_{f\in H^s(\mathbb{R}^n)}\dim _H\left\{x\in\mathbb{R}^n: \lim_{t\rightarrow0}\mathcal{B}_tf(x)\neq f(x)\right\}.
\end{align}
\end{enumerate}

In order to resolve this issue, for any $f\in H^{s>0}(\mathbb{R}^n)$ we make the following decomposition
$$f(x)=f_{<1}(x)+f_{\geq1}(x)\ \ \&\ \ \widehat{f_{<1}}(\xi)=\hat{f}(\xi)\phi(\xi),$$
where $\phi$ is a bump function based on the origin-centered ball $B(0,2)$ with radius $2$ and satisfies
$\phi|_{B(0,1)}=1$.
 Thus $$f_{<1}\in H^s(\mathbb{R}^n)\ \ \forall\ \ s>\frac n2
 ,$$
 which implies
$$\lim_{t\rightarrow0}\mathcal{B}_t(f_{<1})(x)=f_{<1}(x).$$
Accordingly, \eqref{Convergence} amounts to
$$\lim_{t\rightarrow0}\mathcal{B}_t(f_{\geq1})(x)=f_{\geq1}(x), \quad\text{a.e.}\ \  x\in\mathbb{R}^n.$$
Upon noticing
$$|\xi|\sqrt{1+|\xi|^2}\approx |\xi|^2\quad\forall\quad |\xi|\geq 1,$$
we may guess that the Boussinesq operator should behave like the Schr\"odinger operator. However, such a guessing is not easily confirmed. Here are two instances.
\begin{itemize}
	\item [$\rhd$] \cite{DGL, DGLZ,DZ} used the $l^2$-decoupling and polynomial partitioning to set up the following inequality:
$$\Bigg\|\sup_{0<t<1}\left|\mathcal{S}_t(f)(x)\right|\Bigg\|_{L^q\big(B(0,1)\big)}
\lesssim R^{\beta(n,q)}\|f\|_{L^2(\mathbb R^n)},$$
under a suitable condition:

\begin{equation*}
\left\{
\begin{aligned}
 &1<q<\infty;\\
&\beta(n,q)>0; \\
&\text{supp} \hat{f}\subseteq A_R:=\{\xi\in\mathbb{R}^n: |\xi|\approx\ \text{a dyadic number}\ R\}.\\
\end{aligned}
\right.
\end{equation*}
 For this purpose, a scaling argument is essential.  And $\mathcal{S}_t$ is scaling invariant. But $\mathcal{B}_{t}$ does not have this symmetry which is the main difficulty.

 \item[$\rhd$] Cho-Ko \cite{CK} extended  the convergence result on $\mathcal{S}_t$ to some generalized dispersive operators excluding the Boussinesq operator.
 \end{itemize}

But nevertheless, we can still get a first outcome for the one-dimensional case as described below.

\begin{theorem}\label{theorem 1.1}
The optimal order $s_{c,1}=\frac14$ follows from two assertions as seen below.
\begin{enumerate}
  \item[\rm(i)]
  If $s\geq\frac{1}{4}$ and $f\in H^s(\mathbb{R})$, then
  \begin{align}\label{eq:1.1}
 \lim_{t\rightarrow0}\mathcal{B}_tf(x)=f(x),\quad\text{a.e.}\ \  x\in\mathbb{R}.
\end{align}
  \item[\rm(ii)]
  For $0<s<\frac{1}{4}$, there exist two disjoint compact intervals $I, J\subset\mathbb{R}$ and a function $f_0\in H^s(\mathbb{R})$ supported in $I$
  such that
   $$\lim_{t\rightarrow0}\mathcal{B}_tf_0(x)=0 \quad \forall x\in J\ \ \text{fails}.
   $$
\end{enumerate}
\end{theorem}

In order to state the second result on $\mathbb R$, we are required to introduce three more concepts.

\begin{itemize}
	\item[$\rhd$] If $\psi(r):=e^{-r^2}$, then
$$ \mathcal{B}_t^Nf(x):=(2\pi)^{-n}\int_{\mathbb{R}^n} \psi\left(\frac{|\xi|}{N}\right)e^{ix\cdot\xi}e^{it|\xi|\sqrt{1+|\xi|^2}}\hat{f}(\xi)\,\textrm{d}\xi
$$
is called the truncated operator.

\item[$\rhd$]
A positive Borel measure $\mu$ is \emph{$(0,n]\ni\alpha$-dimensional} provided
$$c_\alpha(\mu):=\sup_{(x,r)\in\mathbb{R}^n\times(0,\infty)}r^{-\alpha}{\mu\big(B(x,r)\big)}<\infty\ \ \text{with}\ \ B(x,r)=\{y\in\mathbb R^n: |y-x|<r\}.
$$
Moreover, $M^\alpha(\mathbb{B}^n)$ is the class of the $\alpha$-dimensional probability measures supported in the unit ball $\mathbb{B}^n=B(0,1)$.

\item[$\rhd$] If $A\lesssim B$ stands for $A\le c B$ for a constant $c>0$ and $A\approx B$ means  $A\lesssim B\lesssim A$, then
\begin{align}\label{0.2}
\bar{\alpha}_{n,\mathcal{B}}(s)
:=\inf\left\{\alpha: \frac{\left\|\sup_{(k,N)\in\mathbb N^2}\left|\mathcal{B}_{t_k}^Nf\right|\right\|_{L^1(d\mu)}}{\sqrt{c_\alpha(\mu)}\|f\|_{H^s(\mathbb{R}^n)}}\lesssim 1
\quad \forall (\mu,f,t_k)\in M^\alpha(\mathbb{B}^n)\times H^s(\mathbb{R}^n)\times(0,\infty)\right\}.
\end{align}
If there is no such $\alpha$, we say that
$\bar{\alpha}_{n,\mathcal{B}}(s)$ does not exist.
From \cite{BBCR} and an equivalence of the Hausdorff capacity in \cite[Section 3]{Ad} it follows that
\begin{align}\label{0.3}
\dim _H\left\{x\in\mathbb{R}^n: \lim_{k\rightarrow\infty}\mathcal{B}_{t_k}^Nf(x)\neq f(x)\right\}\leq \bar{\alpha}_{n,\mathcal{B}}(s)\ \ \text{as}\ \ f\in H^s(\mathbb{R}^n)\ \ \&\ \   t_k\rightarrow0.
\end{align}
\end{itemize}

\begin{theorem}\label{theorem 1.10}
Let
\begin{equation*}
\left\{
\begin{aligned}
 &\frac{1}{4}\leq s\leq\frac{1}{2};\\
&\alpha>1-2s; \\
&\mu\in M^\alpha(\mathbb{B});\\
&\lim_{k\rightarrow\infty}t_k=0.\\
\end{aligned}
\right.
\end{equation*}
Then:
\begin{itemize}
\item[\rm(i)] \begin{equation}\label{maximal estimate}
\left\|\sup_{(k,N)\in\mathbb N^2}\left|\mathcal{B}_{t_k}^Nf\right|\right\|_{L^1(d\mu)}\lesssim\sqrt{c_\alpha(\mu)}\|f\|_{H^s(\mathbb{R})};
\end{equation}
\item[\rm (ii)]
 $\bar{\alpha}_{1,\mathcal{B}}(s)=1-2s=\alpha_{1,\mathcal{B}}(s).$
 \end{itemize}
 \end{theorem}

While extending Theorems \ref{theorem 1.1}-\ref{theorem 1.10} to a higher dimension, we are suggested by  \cite{N} (handling the Carleson problem for certain generalized dispersive equation) to consider the radial case, thereby discovering the following assertion.

\begin{theorem}\label{theorem 1.7}
For
\begin{equation*}
\left\{
\begin{aligned}
&n\geq2;\\
&\frac{1}{4}\leq s<\frac{1}{2};\\
&2\leq q\leq\frac{2}{1-2s};\\
&x\in\mathbb R^n;\\
&f\in H^s(\mathbb{R}^n),\\
\end{aligned}
\right.
\end{equation*}
let
$$\mathcal{B}^{\ast\ast}f(x):=\sup _{t\in\mathbb{R}}\left|\mathcal{B}_tf(x)\right|.
$$
Then
\begin{align}\label{1.11}
\left(\int_{\mathbb{R}^n}\left|\mathcal{B}^{\ast\ast}f(x)\right|^q|x|^\alpha\,\textrm{d}x\right)^{\frac{1}{q}}
\lesssim\|f\|_{\dot{H}^s(\mathbb{R}^n)}\ \ \forall\ \ \text{radial}\ \ f\in H^s(\mathbb{R}^n)
\end{align}
if and only if
$$\alpha =q\Big(\frac{n}{2}-s\Big)-n.$$
Consequently, under this situation we have
\begin{align}\label{convergence}
\lim_{t\rightarrow0}\mathcal{B}_tf(x)=f(x), \ \text{a.e.}\ x\in\mathbb R^n.
\end{align}
\end{theorem}

The rest of this paper is designed below to present a much-more-involved proof of the above three theorems according to a twofold argument style.

\section{Demonstration of Theorem \ref{theorem 1.1}}\label{s2}

\subsection{Proof of Theorem \ref{theorem 1.1} (i)}

First of all, recall the following well-known variant of van der Corput's lemma.
\begin{lemma}\label{lemma 4.1}[\cite{Stein} P. 332-334]
For $a<b$ and $I=[a,b]$ let $\Phi\in C^\infty(I)$ be real-valued and $\psi\in C^\infty(I)$.
\begin{enumerate}
  \item[\rm(i)] If $|\Phi'(x)|\geq \gamma>0\ \forall\ x\in I$ and $\Phi'$ is monotonic on $I$, then
  \begin{align*}
  \left|\int_I e^{i\Phi(x)}\psi(x)\,\textrm{d}x\right|
  \lesssim\frac{1}{\gamma}\left(\left|\psi(b)\right|+\int_I \left|\psi'(x)\right|\,\textrm{d}x\right).
  \end{align*}
  \item[\rm(ii)] If $|\Phi''(x)|\geq \gamma>0\ \forall\ x\in I$, then
  \begin{align*}
  \left|\int_I e^{i\Phi(x)}\psi(x)\,\textrm{d}x\right|
  \lesssim\frac{1}{\gamma^{\frac{1}{2}}}\left(\left|\psi(b)\right|+\int_I \left|\psi'(x)\right|\,\textrm{d}x\right).
  \end{align*}
\end{enumerate}
\end{lemma}

Next, we need a crucial oscillatory estimate whose fractional order Schr\"odinger operator analogue was considered in Sj\"olin \cite{S2}.

\begin{lemma}\label{lemma 4.2}
Let
\begin{equation*}
\left\{
\begin{aligned}
 &\frac{1}{2}\leq s<1;\\
&0<t<\frac{1}{6};\\
&q\in C_0^\infty(\mathbb{R}).\\
\end{aligned}
\right.
\end{equation*}
Then
\begin{align*}
  \left|\int_{\mathbb{R}} e^{i(x\cdot\xi+t|\xi|\sqrt{1+|\xi|^2})}|\xi|^{-s}q\left(\frac{\xi}{N}\right)\,\textrm{d}\xi\right|
  \lesssim\frac{1}{|x|^{1-s}}
\end{align*}
for $x\in\mathbb{R}$ and $N>1$.
\end{lemma}

\begin{proof}[Proof of Lemma \ref{lemma 4.2}]
Without loss of generality, we may assume $x\neq0$.
\begin{align*}
 \int_{\mathbb{R}} e^{i(x\cdot\xi+t|\xi|\sqrt{1+|\xi|^2})}|\xi|^{-s}q\left(\frac{\xi}{N}\right)\,\textrm{d}\xi
&=\int_{|\xi|<|x|^{-1}} e^{i(x\cdot\xi+t|\xi|\sqrt{1+|\xi|^2})}|\xi|^{-s}q\left(\frac{\xi}{N}\right)\,\textrm{d}\xi\\
&\quad +\int_{|\xi|\geq|x|^{-1}} e^{i(x\cdot\xi+t|\xi|\sqrt{1+|\xi|^2})}|\xi|^{-s}q\left(\frac{\xi}{N}\right)\,\textrm{d}\xi\\
&:=A+B.
\end{align*}

It is easy to estimate
\begin{align*}
 |A|\lesssim\int_{|\xi|<|x|^{-1}}|\xi|^{-s}\,\textrm{d}\xi\thickapprox\frac{1}{|x|^{1-s}}.
\end{align*}
However, estimating $B$ is divided into two cases.
\begin{itemize}
\item[$\rhd$] {\bf Case~1}: $|x|<1$. Two subcases are considered below.
\begin{enumerate}
  \item[\rm(i)] Under $|x|^2\leq\frac{t}{2}$,
  we set
  $$\Phi(\xi):=x\cdot\xi+t\xi\sqrt{1+\xi^2}\quad \forall\xi\geq|x|^{-1}.$$ Then
 $$\Phi'(\xi)=x+t\frac{1+2\xi^2}{\sqrt{1+\xi^2}}=x\left(1+\frac{t}{x}\frac{1+2\xi^2}{\sqrt{1+\xi^2}}\right).$$
 Note that
 $$\left|1+\frac{t}{x}\frac{1+2\xi^2}{\sqrt{1+\xi^2}}\right|
 \geq\left|\frac{t}{x}\frac{1+2\xi^2}{\sqrt{1+\xi^2}}\right|-1
 \geq\frac{2|x|^2}{|x|}\frac{2\xi^2}{\sqrt{2}\xi}-1
 \geq2\sqrt{2}|x|\xi-1
 \geq1.$$
 So,
 $|\Phi'(\xi)|\geq|x|.$

Also, we have
 $$\Phi''(\xi)=t\frac{3\xi+2\xi^3}{(1+\xi^2)^{\frac{3}{2}}},
 $$
  thereby getting that $\Phi'$ is monotonic for $\xi\geq|x|^{-1}$. Upon setting
 $$\psi(\xi):=\xi^{-s}q\left(\frac{\xi}{N}\right),
 $$
  we obtain
  $$|\psi(\xi)|\lesssim|x|^{-s}\quad\forall\xi\geq|x|^{-1}.$$
 \begin{align*}
 \int_{|x|^{-1}}^\infty |\psi'(\xi)|\,\textrm{d}\xi
 &=\int_{|x|^{-1}}^\infty \left|\xi^{-s}\frac{1}{N}q'\left(\frac{\xi}{N}\right)-s\xi^{-s-1}q\left(\frac{\xi}{N}\right)\right|\,\textrm{d}\xi\\
 &\lesssim|x|^s\int_{|x|^{-1}}^\infty \frac{1}{N}\left|q'\left(\frac{\xi}{N}\right)\right|\,\textrm{d}\xi+\int_{|x|^{-1}}^\infty\xi^{-s-1}\,\textrm{d}\xi\\
 &\lesssim|x|^s.
\end{align*}
By Lemma \ref{lemma 4.1}, we then conclude that
$$
\begin{cases}\left|\int_{|x|^{-1}}^\infty e^{i(x\cdot\xi+t|\xi|\sqrt{1+|\xi|^2})}|\xi|^{-s}q\left(\frac{\xi}{N}\right)\,\textrm{d}\xi\right|
\lesssim\frac{1}{|x|}|x|^s
\thickapprox\frac{1}{|x|^{1-s}};\\
\left|\int_{-\infty}^{-|x|^{-1}} e^{i(x\cdot\xi+t|\xi|\sqrt{1+|\xi|^2})}|\xi|^{-s}q\left(\frac{\xi}{N}\right)\,\textrm{d}\xi\right|
\lesssim\frac{1}{|x|}|x|^s
\thickapprox\frac{1}{|x|^{1-s}}.
\end{cases}
$$

\item[\rm(ii)] Under $|x|^2>\frac{t}{2}$, we achieve
$$
\int_{|x|^{-1}}^\infty e^{i(x\cdot\xi+t|\xi|\sqrt{1+|\xi|^2})}|\xi|^{-s}q\left(\frac{\xi}{N}\right)\,\textrm{d}\xi
:=B_1+B_2+B_3
$$
where
$$
\begin{cases}
B_1:=\int_{I_1}e^{i(x\cdot\xi+t|\xi|\sqrt{1+|\xi|^2})}|\xi|^{-s}q\left(\frac{\xi}{N}\right)\,\textrm{d}\xi;\\
B_2:=\int_{I_2}e^{i(x\cdot\xi+t|\xi|\sqrt{1+|\xi|^2})}|\xi|^{-s}q\left(\frac{\xi}{N}\right)\,\textrm{d}\xi;\\
B_3:=\int_{I_3}e^{i(x\cdot\xi+t|\xi|\sqrt{1+|\xi|^2})}|\xi|^{-s}q\left(\frac{\xi}{N}\right)\,\textrm{d}\xi,
\end{cases}
$$
and
\begin{equation*}
\left\{
\begin{aligned}
 &I_1:=\left\{\xi\geq|x|^{-1},\xi\leq\delta\frac{|x|}{t}\right\};\\
&I_2:=\left\{\xi\geq|x|^{-1},\delta\frac{|x|}{t}\leq\xi\leq K\frac{|x|}{t}\right\};\\
&I_3:=\left\{\xi\geq|x|^{-1},\xi\geq K\frac{|x|}{t}\right\};\\
&\delta>0\ \&\ K>0\  \text{are small \& large numbers respectively}.\\
\end{aligned}
\right.
\end{equation*}

For $\xi\in I_1$ we have
$$t\frac{1+2\xi^2}{\sqrt{1+\xi^2}}\leq t\frac{3\xi^2}{\xi}\leq 3t\delta\frac{|x|}{t}\leq\frac{1}{2}|x|,$$
 whence
 $$|\Phi'(\xi)|\geq|x|-t\frac{1+2\xi^2}{\sqrt{1+\xi^2}}\geq|x|-\frac{1}{2}|x|=\frac{1}{2}|x|.$$
By Lemma \ref{lemma 4.1}, we can obtain
 $$|B_1|
 \lesssim\frac{1}{|x|^{1-s}}.$$
 Meanwhile, upon estimating
 $$t\frac{1+2\xi^2}{\sqrt{1+\xi^2}}\geq t\frac{2\xi^2}{\sqrt{2}\xi}\geq \sqrt{2}tK\frac{|x|}{t}\geq 2|x|\ \ \forall\ \
 \xi\in I_3,
 $$
 we get
 $$|\Phi'(\xi)|\geq t\frac{1+2\xi^2}{\sqrt{1+\xi^2}}-|x|\geq|x|\quad\forall \xi \in I_3.
 $$
 By Lemma \ref{lemma 4.1}, we have
  $$|B_3|\lesssim\frac{1}{|x|^{1-s}}.$$

 To estimate $I_2$, note that
 $$
 \begin{cases}\Phi''(\xi)=t\frac{3\xi+2\xi^3}{(1+\xi^2)^{\frac{3}{2}}}\geq t\frac{2\xi^3}{(2\xi^2)^{\frac{3}{2}}}=\frac{\sqrt{2}}{2}t \quad\forall \xi\in I_2;\\
 |\psi(\xi)|
 =\left|\xi^{-s}q\left(\frac{\xi}{N}\right)\right|
 \lesssim\left(\delta\frac{|x|}{t}\right)^{-s}
 \lesssim\left(\frac{|x|}{t}\right)^{-s},
 \end{cases}
 $$
 and
 \begin{align*}
 \int_{I_2}|\psi'(\xi)|\,\textrm{d}\xi
 &=\int_{I_2} \left|\xi^{-s}\frac{1}{N}q'\left(\frac{\xi}{N}\right)-s\xi^{-s-1}q\left(\frac{\xi}{N}\right)\right|\,\textrm{d}\xi\\
 &\leq\int_{I_2}|\xi|^{-s} \frac{1}{N}\left|q'\left(\frac{\xi}{N}\right)\right|\,\textrm{d}\xi+C\int_{I_2}\xi^{-s-1}\,\textrm{d}\xi\\
 &\leq\left(\delta\frac{|x|}{t}\right)^{-s}\int_{I_2}\frac{1}{N}\left|q'\left(\frac{\xi}{N}\right)\right|\,\textrm{d}\xi+C\left(\frac{|x|}{t}\right)^{-s}\\
 &\lesssim\left(\frac{|x|}{t}\right)^{-s}.
\end{align*}
So, by Lemma \ref{lemma 4.1} and $s\ge \frac12$, we obtain
 $$|B_2|
 \lesssim t^{-\frac{1}{2}}\left(\frac{|x|}{t}\right)^{-s}
 \thickapprox t^{s-\frac{1}{2}}|x|^{-s}
 \lesssim|x|^{2(s-\frac{1}{2})}|x|^{-s}
 \thickapprox\frac{1}{|x|^{1-s}},$$
 thereby finding
$$
\begin{cases}\left|\int_{|x|^{-1}}^\infty e^{i(x\cdot\xi+t|\xi|\sqrt{1+|\xi|^2})}|\xi|^{-s}q\left(\frac{\xi}{N}\right)\,\textrm{d}\xi\right|
\lesssim\frac{1}{|x|^{1-s}};\\
\left|\int_{-\infty}^{-|x|^{-1}} e^{i(x\cdot\xi+t|\xi|\sqrt{1+|\xi|^2})}|\xi|^{-s}q\left(\frac{\xi}{N}\right)\,\textrm{d}\xi\right|
\lesssim\frac{1}{|x|^{1-s}}.
\end{cases}
$$
\end{enumerate}

\item[$\rhd$] {\bf Case~2}: $|x|>1$. Upon writing
\begin{align*}
\int_{|x|^{-1}}^\infty e^{i(x\cdot\xi+t|\xi|\sqrt{1+|\xi|^2})}|\xi|^{-s}q\left(\frac{\xi}{N}\right)\,\textrm{d}\xi
&:=\int_{|x|^{-1}}^1 e^{i(x\cdot\xi+t|\xi|\sqrt{1+|\xi|^2})}|\xi|^{-s}q\left(\frac{\xi}{N}\right)\,\textrm{d}\xi\\
&\quad +\int_1^\infty e^{i(x\cdot\xi+t|\xi|\sqrt{1+|\xi|^2})}|\xi|^{-s}q\left(\frac{\xi}{N}\right)\,\textrm{d}\xi\\
&:=B_4+B_5.
\end{align*}
we utilize the technique similar to handling {\bf Case~1} to gain
$$|B_5|\lesssim\frac{1}{|x|^{1-s}}.$$

Concerning $B_4$, we estimate
$$
\begin{cases} |\Phi'(\xi)|
=\left|x+t\frac{1+2\xi^2}{\sqrt{1+\xi^2}}\right|
\geq|x|\left(1-\left|\frac{t}{x}\frac{1+2\xi^2}{\sqrt{1+\xi^2}}\right|\right)
\geq\frac{|x|}{2}\ \ \forall\ \
0<t<\frac{1}{6};\\
|\psi(\xi)|
=\left|\xi^{-s}q\left(\frac{\xi}{N}\right)\right|
\lesssim\xi^{-s}
\lesssim|x|^{s}\ \ \forall\ \ \xi>|x|^{-1},
\end{cases}
$$
and
 \begin{align*}
 \int_{|x|^{-1}}^1|\psi'(\xi)|\,\textrm{d}\xi
 &=\int_{|x|^{-1}}^1 \left|\xi^{-s}\frac{1}{N}q'\left(\frac{\xi}{N}\right)-s\xi^{-s-1}q\left(\frac{\xi}{N}\right)\right|\,\textrm{d}\xi\\
 &\leq\int_{|x|^{-1}}^1\xi^{-s} \frac{1}{N}\left|q'\left(\frac{\xi}{N}\right)\right|\,\textrm{d}\xi+C\int_{|x|^{-1}}^1\xi^{-s-1}\,\textrm{d}\xi\\
 &\lesssim|x|^{s}\int_{|x|^{-1}}^1\frac{1}{N}\left|q'\left(\frac{\xi}{N}\right)\right|\,\textrm{d}\xi+|x|^s\\
 &\lesssim|x|^s.
\end{align*}
By Lemma \ref{lemma 4.1}, we have
$$|B_4|=\left|\int_{|x|^{-1}}^1 e^{i(x\cdot\xi+t|\xi|\sqrt{1+|\xi|^2})}|\xi|^{-s}q(\frac{\xi}{N})\,\textrm{d}\xi\right| \lesssim\frac{1}{|x|^{1-s}},
$$
thereby reaching
$$
\begin{cases}\left|\int_{|x|^{-1}}^\infty e^{i(x\cdot\xi+t|\xi|\sqrt{1+|\xi|^2})}|\xi|^{-s}q\left(\frac{\xi}{N}\right)\,\textrm{d}\xi\right|
\lesssim\frac{1}{|x|^{1-s}};\\
\left|\int_{-\infty}^{-|x|^{-1}} e^{i(x\cdot\xi+t|\xi|\sqrt{1+|\xi|^2})}|\xi|^{-s}q\left(\frac{\xi}{N}\right)\,\textrm{d}\xi\right|
\lesssim\frac{1}{|x|^{1-s}}.
\end{cases}
$$
\end{itemize}
\end{proof}

Finally, we arrive at

\begin{proof}[Proof of Theorem \ref{theorem 1.1} (i)]
It is enough to prove that the following estimate
\begin{equation}\label{2.1}
\left(\int_B \left|\mathcal{B}^{\ast} f(x)\right|^2\,\textrm{d}x\right)^{\frac{1}{2}}
\lesssim\|f\|_{{\dot H}^s{(\mathbb{R})}},\quad f\in {\dot H}^s(\mathbb{R})
\end{equation}
holds for all balls $B$ in $\mathbb{R}$ due to the fact
 that \eqref{2.1} implies Theorem \ref{theorem 1.1}(i) and
 $$\mathcal{B}^{\ast}f(x):=\sup _{0<t<1}\left|\mathcal{B}_tf(x)\right|.$$

To do so, set $$Rf(x):=\phi(x)\int_\mathbb{R}e^{ix\cdot\xi}e^{it(x)|\xi|\sqrt{1+\xi^2}}|\xi|^{-s}\hat{f}(\xi)\,\textrm{d}\xi,\quad x\in\mathbb{R},\quad f\in S(\mathbb{R}),$$
where $\phi(x)$ is a real-valued function in $C_c^\infty(\mathbb{R})$, $t(x)$ is a measurable function of $x$ with $0<t(x)<1$ and $s=\frac{1}{4}$. It suffices to prove that the operator $R$ is bounded on $L^2(\mathbb{R})$.

Upon letting
$$p(x,\xi):=\phi(x)e^{it(x)|\xi|\sqrt{1+\xi^2}}|\xi|^{-s},$$
we get
$$Rf(x)=\int_\mathbb{R}e^{ix\cdot\xi}p(x,\xi)\hat{f}(\xi)\,\textrm{d}\xi.$$
Via choosing a real-valued function $\rho\in C_c^\infty(\mathbb{R})$ such that
\begin{equation*}
\left\{
\begin{aligned}
 &\rho(\xi)=
\left\{
\begin{aligned}
 &1 \quad \textrm{as} \quad |\xi|\leq 1;\\
&0 \quad \textrm{as} \quad |\xi|\geq 2; \\
\end{aligned}
\right.\\
&\rho_N(\xi):=\rho\left(\frac{\xi}{N}\right) \quad \forall N\in \mathbb{N},\\
\end{aligned}
\right.
\end{equation*}
and defining
$$
\begin{cases}
p_N(x,\xi):=\rho_N(\xi)p(x,\xi);\\ $$R_Nf(x):=\int_\mathbb{R}e^{ix\cdot\xi}p_N(x,\xi)\hat{f}(\xi)\,\textrm{d}\xi,\quad f\in S(\mathbb{R}).
\end{cases}
$$
we can readily see that $R_N$ is bounded on $L^2(\mathbb{R})$ and its adjoint operator is given by
$$R_N^\ast h(x)=\int_\mathbb{R}\int_\mathbb{R}e^{i(x-y)\cdot\xi}\overline{p_N(y,\xi)}h(y)\,\textrm{d}y\,\textrm{d}\xi,\quad h\in S(\mathbb{R}).
$$
According to Plancherel's theorem, we have
\begin{eqnarray*}
\int_{\mathbb{R}}\left|R_N^\ast h(x)\right|^2\,\textrm{d}x
&=&\int_{\mathbb{R}}\left|\widehat{R_N^\ast h}(\xi)\right|^2\,\textrm{d}\xi\\
&=&\int_{\mathbb{R}}\widehat{R_N^\ast h}(\xi)\overline{\widehat{R_N^\ast h}(\xi)}\,\textrm{d}\xi\\
&=&(2\pi)^2\int_{\mathbb{R}}\int_{\mathbb{R}}e^{-i y\cdot\xi}\overline{p_N(y,\xi)}h(y)\,\textrm{d}y\overline{\int_{\mathbb{R}}e^{-iz\cdot\xi}\overline{p_N(z,\xi)}h(z)\,\textrm{d}z}\,\textrm{d}\xi\\
&=&(2\pi)^2\int_{\mathbb{R}}\int_{\mathbb{R}}e^{-i y\cdot\xi}\overline{p_N(y,\xi)}h(y)\,\textrm{d}y\int_{\mathbb{R}}e^{iz\cdot\xi}p_N(z,\xi)\overline{h(z)}\,\textrm{d}z\,\textrm{d}\xi\\
&=&(2\pi)^2\int_{\mathbb{R}}\int_{\mathbb{R}}\int_{\mathbb{R}}e^{i[(z-y)\cdot\xi+(t(z)-t(y))|\xi|\sqrt{1+\xi^2}]}
\rho_N^2(\xi)|\xi|^{-2s}\,\textrm{d}\xi\phi(y)\phi(z)h(y)\overline{h(z)}\,\textrm{d}y\,\textrm{d}z\\
&=&(2\pi)^2\int_{\mathbb{R}}\int_{\mathbb{R}}K_N(y, z)h(y)\overline{h(z)}\,\textrm{d}y\,\textrm{d}z,
\end{eqnarray*}
where $$K_N(y, z):=\phi(y)\phi(z)\int_{\mathbb{R}}e^{i((z-y)\cdot\xi+(t(z)-t(y))|\xi|\sqrt{1+\xi^2})}
\rho_N^2(\xi)|\xi|^{-2s}\,\textrm{d}\xi\ \ \&\ \ s=\frac14.
$$
By Lemma \ref{lemma 4.2}, we have
\begin{align}\label{2.2}
|K_N(y, z)|\lesssim|\phi(y)||\phi(z)||y-z|^{-\frac{1}{2}},\quad N>1,
\end{align}
thereby getting
\begin{eqnarray*}
\int_{\mathbb{R}}\left|R_N^\ast h(x)\right|^2\,\textrm{d}x
&=&(2\pi)^2\int_{\mathbb{R}}\int_{\mathbb{R}}K_N(y, z)h(y)\overline{h(z)}\,\textrm{d}y\,\textrm{d}z\\
&\lesssim&\int_{\mathbb{R}}\int_{\mathbb{R}}\left|y-z\right|^{-\frac{1}{2}}\left|\phi(y)\phi(z)\right|\left|h(y)h(z)\right|\,\textrm{d}y\,\textrm{d}z\\
&\thickapprox&\int_{\mathbb{R}}\int_{\mathbb{R}}\frac{\left|\phi(z)h(z)\right|}{\left|y-z\right|^{\frac{1}{2}}}\,\textrm{d}z\left|\phi(y)h(y)\right|\,\textrm{d}y\\
&\thickapprox&\int_{\mathbb{R}}I_{\frac{1}{2}}\left(\left|\phi h\right|\right)(y)\left|\phi(y)h(y)\right|\,\textrm{d}y\\
&\lesssim&\|I_{\frac{1}{2}}(|\phi h|)\|_{L^4(\mathbb{R})}\|\phi h\|_{L^{\frac{4}{3}}(\mathbb{R})}\\
&\lesssim&\|\phi h\|^2_{L^{\frac{4}{3}}(\mathbb{R})}\\
&\lesssim&\|\phi\|^2_{L^4(\mathbb{R})}\|h\|^2_{L^2(\mathbb{R})}\\
&\lesssim&\|h\|^2_{L^2(\mathbb{R})},
\end{eqnarray*}
where $I_{\frac{1}{2}}$ denotes the Riesz potential operator of order $\frac{1}{2}$ which is bounded from $L^{\frac{4}{3}}(\mathbb{R})$ to $L^4(\mathbb{R})$, and H\"{o}lder's inequality has been utilized two times. Accordingly, $R_N$ is uniformly bounded on $L^2(\mathbb{R})$. Since
$$Rf(x)=\lim_{N\rightarrow\infty}R_Nf(x),\quad f\in S(\mathbb{R}),$$
 we use Fatou's lemma  to get
$$\int_{\mathbb{R}}|Rf(x)|^2\,\textrm{d}x
\leq\underline{\lim}_{N\rightarrow\infty}\int_{\mathbb{R}}|R_Nf(x)|^2\,\textrm{d}x
\lesssim\|f\|^2_{L^2(\mathbb{R})},
$$
thereby establishing $L^2(\mathbb{R})$-boundedness of $R$.
\end{proof}
\subsection{Proof of Theorem \ref{theorem 1.1} (ii)}

To verify Theorem \ref{theorem 1.1} (ii), we need four lemmas.

Firstly, we are about to use the fact that the dispersive waves with different frequencies transport at different velocities - more clearly - we take
$$
\begin{cases} g\in S(\mathbb{R});\\
$$\text{supp} \check{g}\subset(-1,1);\\
 \check{g}(0)\neq0;\\
 f_v(x):=e^{-\frac{ix}{v^2}}\check{g}\left(\frac{x}{v}\right),\quad 0<v<1.
 \end{cases}
 $$
Note that
$$\text{supp} f_v\subset(-v, v)\ \ \&\ \ f_v\in S(\mathbb{R}).$$
So, we have
\begin{lemma}\label{lemma 2.1}
(\cite{DK})
If $$0<s<\frac{1}{4}\ \ \&\ \ 0<v<1,$$
then $$\widehat{f_v}(\xi)=v g\left(v\xi+\frac{1}{v}\right)\ \ \&\ \
\|f_v\|_{H^s(\mathbb{R})}\lesssim v^{\frac{1}{2}-2s}.$$
\end{lemma}
Secondly, this last lemma, along with Heisenberg's inequality, leads to
\begin{lemma}\label{lemma 2.2}
For $$\Phi(\xi):=|\xi|\sqrt{1+\xi^2},$$ there exists a triple $\{c_0>0,\delta>0,v_0>0\}$ such that $$\left|\mathcal{B}_t f_v(x)\right|\geq c_0$$
holds
for any triple $\{v,x,t\}$ with
\begin{equation*}
\left\{
\begin{aligned}
 &0<v<v_0;\\
&0<x<\delta;\\
&t=\frac{x}{\Phi'\left(\frac{1}{v^2}\right)}.\\
\end{aligned}
\right.
\end{equation*}
\end{lemma}
\begin{proof}[Proof of Lemma \ref{lemma 2.2}] Firstly,
since
$$\check{g}(0)\neq 0
,\Rightarrow \int_\mathbb{R} g(\xi)\,\textrm{d}\xi\neq 0,$$
we choose a large number L such that
$$\int_{|\xi|\geq L}\left|g(\xi)\right|\,\textrm{d}\xi\leq \frac{1}{100}\left|\int_\mathbb{R} g(\xi)\,\textrm{d}\xi\right|,$$
thereby evaluating
\begin{align*}
\mathcal{B}_tf_v(x)
&=\int_\mathbb{R} e^{i(x\cdot\xi+t\Phi(\xi))}\widehat{f_v}(\xi)\,\textrm{d}\xi\\
&=\int_\mathbb{R} e^{i(x\cdot\xi+t\Phi(\xi))}v g\left(v\xi+\frac{1}{v}\right)\,\textrm{d}\xi\\
&=\int_\mathbb{R} e^{i\left(\frac{x}{v}\left(\xi-\frac{1}{v}\right)+t\Phi\left(\frac{1}{v}\left|\xi-\frac{1}{v}\right|\right)\right)}g(\xi)\,\textrm{d}\xi\\
&:=\int_\mathbb{R} e^{iF}g(\xi)\,\textrm{d}\xi,
\end{align*}
where $$F(x, t,\xi,v):=\frac{x}{v}\left(\xi-\frac{1}{v}\right)+t\Phi\left(\frac{1}{v}\left|\xi-\frac{1}{v}\right|\right).$$
Via choosing $v_0:=\frac{1}{2L}$, we get that if $0<v<v_0$ then
\begin{align*}
\left|\mathcal{B}_tf_v(x)\right|
&\geq\left|\int_{|\xi|\leq L} e^{iF}g(\xi)\,\textrm{d}\xi\right|-\left|\int_{|\xi|\geq L} e^{iF}g(\xi)\,\textrm{d}\xi\right|\\
&\geq\left|\int_{|\xi|\leq L} e^{iF}g(\xi)\,\textrm{d}\xi\right|-\int_{|\xi|\geq L} \left|g(\xi)\right|\,\textrm{d}\xi\\
&\geq\left|\int_{|\xi|\leq L} e^{iF}g(\xi)\,\textrm{d}\xi\right|-\frac{1}{100}\left|\int_\mathbb{R} g(\xi)\,\textrm{d}\xi\right|.
\end{align*}

Next we consider
$$\left|\int_{|\xi|\leq L} e^{iF}g(\xi)\,\textrm{d}\xi\right|.$$
With the help of Taylor's expansion we have
\begin{align*}
\Phi\left(\frac{1}{v}\left|\xi-\frac{1}{v}\right|\right)
&=\Phi\left(\frac{1}{v^2}-\frac{\xi}{v}\right)\\
&=\Phi\left(\frac{1}{v^2}\right)-\frac{\xi}{v}\Phi'\left(\frac{1}{v^2}\right)+\frac{\xi^2}{v^2}\frac{\Phi''\left(\frac{1}{v^2}\right)}{2}+O\left(-\frac{\xi}{v}\right)^3,
\end{align*}
whence
$$F(x, t,\xi,v)=\frac{x\cdot\xi}{v}-\frac{x}{v^2}+t\Phi\left(\frac{1}{v^2}\right)
-t\frac{\xi}{v}\Phi'\left(\frac{1}{v^2}\right)+t\frac{\xi^2}{v^2}\frac{\Phi''\left(\frac{1}{v^2}\right)}{2}+O\left[t\left(-\frac{\xi}{v}\right)^3\right].$$
An application of
$$t:={\frac{x}{\Phi'\left(\frac{1}{v^2}\right)}}$$
yields
$$F(x, t,\xi,v)=-\frac{x}{v^2}+\frac{x\Phi\left(\frac{1}{v^2}\right)}{\Phi'\left(\frac{1}{v^2}\right)}
+\frac{x}{\Phi'\left(\frac{1}{v^2}\right)}\frac{\Phi''\left(\frac{1}{v^2}\right)}{2}\frac{\xi^2}{v^2}
+O\left(-\frac{x}{\Phi'\left(\frac{1}{v^2}\right)}\frac{\xi^3}{v^3}\right).$$
Therefore, for small $\delta>0$ we have
\begin{align*}
\left|\int_{|\xi|\leq L} e^{iF}g(\xi)\,\textrm{d}\xi\right|
&=\left|\int_{-L}^L e^{i\left(\frac{x}{\Phi'\left(\frac{1}{v^2}\right)}\frac{\Phi''\left(\frac{1}{v^2}\right)}{2}\frac{\xi^2}{v^2}\right)}
 e^{iO\left(-\frac{x}{\Phi'\left(\frac{1}{v^2}\right)}\frac{\xi^3}{v^3}\right)} g(\xi)\,\textrm{d}\xi\right|\\
&= \left|\int_{-L}^L  e^{i\left(\frac{x}{\Phi'\left(\frac{1}{v^2}\right)}\frac{\Phi''\left(\frac{1}{v^2}\right)}{2}\frac{\xi^2}{v^2}\right)}g(\xi)\,\textrm{d}\xi
+\int_{-L}^L  e^{i\left(\frac{x}{\Phi'\left(\frac{1}{v^2}\right)}\frac{\Phi''\left(\frac{1}{v^2}\right)}{2}\frac{\xi^2}{v^2}\right)}
 \left[e^{iO\left(-\frac{x}{\Phi'\left(\frac{1}{v^2}\right)}\frac{\xi^3}{v^3}\right)}-1\right] g(\xi)\,\textrm{d}\xi\right|\\
&\gtrsim\left|\int_{-L}^L  e^{i\left(\frac{x}{\Phi'\left(\frac{1}{v^2}\right)}\frac{\Phi''\left(\frac{1}{v^2}\right)}{2}\frac{\xi^2}{v^2}\right)}g(\xi)\,\textrm{d}\xi\right|-x\\
&\geq \frac{1}{2}\left|\int_{-L}^L g(\xi)\,\textrm{d}\xi\right| \quad \forall x\in(0,\delta).
\end{align*}

Finally, if $$0<v<v_0\ \ \&\ \ 0<x<\delta,
$$
then
\begin{align*}
\left|\mathcal{B}_tf_v(x)\right|
&\geq\frac{1}{2}\left|\int_{-L}^L g(\xi)\,\textrm{d}\xi\right|-\frac{1}{100}\left|\int_\mathbb{R} g(\xi)\,\textrm{d}\xi\right|\\
&\geq\frac{1}{2}\left|\int_\mathbb{R} g(\xi)\,\textrm{d}\xi\right|-\frac{1}{200}\left|\int_\mathbb{R} g(\xi)\,\textrm{d}\xi\right|
-\frac{1}{100}\left|\int_\mathbb{R} g(\xi)\,\textrm{d}\xi\right|\\
&\geq\frac{1}{4}\left|\int_\mathbb{R} g(\xi)\,\textrm{d}\xi\right|,
\end{align*}
as desired.
\end{proof}

Thirdly, although the following lemma is known as the dispersion estimate, we will use it to check the wave with high frequency spread so fast such that in the test interval $\frac{\delta}{2}<x<\delta$, the wave is already gone at the given time.

\begin{lemma}\label{lemma 2.3}
If
\begin{equation*}
\left\{
\begin{aligned}
 &0<v<\min\left\{v_0,\frac{\delta}{4}\right\};\\
&0<t<1;\\
&\frac{\delta}{2}<x<\delta,\\
\end{aligned}
\right.
\end{equation*}
then
$$|\mathcal{B}_tf_v(x)|
\lesssim\frac{v}{t^{\frac{1}{2}}}.$$
\end{lemma}
\begin{proof}[Proof of Lemma \ref{lemma 2.3}]
Let
\begin{equation*}
\left\{
\begin{aligned}
 &m(\xi):=e^{it|\xi|\sqrt{1+\xi^2}};\\
&\hat{m}(y):=\int_{\mathbb{R}}e^{it(|\xi|\sqrt{1+\xi^2}-\frac{y\cdot\xi}{t})}\,\textrm{d}\xi.\\
\end{aligned}
\right.
\end{equation*}
Then two situations are handled below.

\begin{enumerate}
  \item[$\rhd$] Under $|\xi|\leq1$, it is easy to see that $$|\hat{m}(y)|\lesssim1\lesssim t^{-\frac{1}{2}}.$$
  \item[$\rhd$] Under $|\xi|>1$, we set
  $$\Psi(\xi):=|\xi|\sqrt{1+\xi^2}-\frac{y\cdot\xi}{t}.$$
If $\xi\geq0$, then
$$\Psi'(\xi)=\frac{1+2\xi^2}{(1+\xi^2)^{\frac{1}{2}}}-\frac{y}{t}~\ \ \&\ \  ~\Psi''(\xi)=\frac{3\xi+2\xi^3}{(1+\xi^2)^{\frac{3}{2}}}.$$
Upon letting
$$\Psi'(\xi_0)=0= \frac{1+2\xi_0^2}{(1+\xi_0^2)^{\frac{1}{2}}}-\frac{y}{t},$$
 we get
 $$
 \begin{cases} t=\frac{y}{\Phi'(\xi_0)}=\frac{x}{\Phi'\left(\frac{1}{v^2}\right)};\\
\xi_0\thickapprox\frac{1}{v^2};\\
\left|\Psi''(\xi_0)\right|\thickapprox\left|\Psi''\left(\frac{1}{v^2}\right)\right|\gtrsim1.
\end{cases}
$$
According to stationary phase method again, we can obtain
$$|\hat{m}(y)|\lesssim\frac{1}{\sqrt{t\Psi''(\xi_0)}}\lesssim t^{-\frac{1}{2}},
$$
whence
\begin{align*}
\left|\mathcal{B}_tf_v(x)\right|
&=\left|\int_{\mathbb{R}} e^{i(x\cdot\xi+t\Phi(\xi))}\widehat{f_v}(\xi)\,\textrm{d}\xi\right|\\
&=\left|\int_{\mathbb{R}} \hat{m}(y)f_v(x-y)\,\textrm{d}y\right|\\
&\leq\int_{\mathbb{R}} \left|\hat{m}(y)\right|\left|f_v(x-y)\right|\,\textrm{d}y\\
&\lesssim\int_{\mathbb{R}} \frac{1}{t^{\frac{1}{2}}}\left|f_v(y)\right|\,\textrm{d}y\\
&\thickapprox\frac{1}{t^{\frac{1}{2}}} \int_{\mathbb{R}} \left|e^{-i\frac{y}{v^2}}\check{g}\left(\frac{y}{v}\right)\right|\,\textrm{d}y\\
&\thickapprox\frac{1}{t^{\frac{1}{2}}} v\int_{\mathbb{R}}\left|\check{g}(s)\right|\,\textrm{d}s\\
&\lesssim\frac{v}{t^{\frac{1}{2}}}.
\end{align*}
\end{enumerate}
\end{proof}
Fourthly, the coming-up-next lemma will be used to check the wave with low frequency which cannot arrive at the test interval $\frac{\delta}{2}<x<\delta$ at given time.
\begin{lemma}\label{lemma 2.4}
If
\begin{equation*}
\left\{
\begin{aligned}
 &0<v<\min\left\{v_0,\frac{\delta}{4}\right\};\\
&0<t<1;\\
&\frac{\delta}{2}<x<\delta,\\
\end{aligned}
\right.
\end{equation*}
then
$$|\mathcal{B}_tf_v(x)|\lesssim\frac{t}{v^4}.$$
\end{lemma}
\begin{proof}[Proof of Lemma \ref{lemma 2.4}]
According to the definition of $f_v(x)$, we have
\begin{align*}
\mathcal{B}_tf_v(x)
&=\int_{\mathbb{R}} e^{ix\cdot\xi}e^{it|\xi|\sqrt{1+\xi^2}}\widehat{f_v}(\xi)\,\textrm{d}\xi\\
&=\int_{\mathbb{R}} \left(e^{it|\xi|\sqrt{1+\xi^2}}-1\right) e^{ix\cdot\xi}\widehat{f_v}(\xi)\,\textrm{d}\xi
+2\pi f_v(x),
\end{align*}
thereby getting $f_v(x)=0$ due to
\begin{equation*}
\left\{
\begin{aligned}
 &\text{supp} f_v\subset(-v, v);\\
&0<v<\frac{\delta}{4};\\
&\frac{\delta}{2}<x<\delta.\\
\end{aligned}
\right.
\end{equation*}
Consequently,
\begin{align*}
\left|\mathcal{B}_tf_v(x)\right|
&=\left|\int_{\mathbb{R}} \left(e^{it|\xi|\sqrt{1+\xi^2}}-1\right) e^{ix\cdot\xi}\widehat{f_v}(\xi)\,\textrm{d}\xi\right|\\
&\lesssim\int_{|\xi|\leq1} t|\xi|\sqrt{1+\xi^2}\left|\widehat{f_v}(\xi)\right|\,\textrm{d}\xi
+\int_{|\xi|>1} t|\xi|\sqrt{1+\xi^2}\left|\widehat{f_v}(\xi)\right|\,\textrm{d}\xi\\
&:=I_{1}+I_{2}.
\end{align*}
On the one hand, we have
\begin{align*}
I_{1}
&\lesssim\int_{|\xi|\leq1} t|\xi|\left|v g\left(v\xi+\frac{1}{v}\right)\right|\,\textrm{d}\xi\\
&\lesssim\frac{t}{v}\int_{\mathbb{R}}\left|\eta-\frac{1}{v}\right|\left|g(\eta)\right|\,\textrm{d}\eta\\
&\lesssim\frac{t}{v}\int_{\mathbb{R}}\left(|\eta|+\frac{1}{v}\right)\left|g(\eta)\right|\,\textrm{d}\eta\\
&\lesssim\frac{t}{v^2}.
\end{align*}
On the other hand, we have
\begin{align*}
I_{2}
&\lesssim\int_{|\xi|>1} t|\xi|^2\left|v g\left(v\xi+\frac{1}{v}\right)\right|\,\textrm{d}\xi\\
&\lesssim\frac{t}{v^2}\int_{\mathbb{R}}\left|\eta-\frac{1}{v}\right|^2|g(\eta)|\,\textrm{d}\eta\\
&\lesssim\frac{t}{v^2}\int_{\mathbb{R}}\left(|\eta|^2+\frac{1}{v^2}\right)|g(\eta)|\,\textrm{d}\eta\\
&\lesssim\frac{t}{v^4}.
\end{align*}
Combining the  estimate of $I_{1}$ with $I_{2}$, we have
$$|\mathcal{B}_tf_v(x)|\lesssim\frac{t}{v^4}.$$
\end{proof}

With the aid of the previous four lemmas, we come to
\begin{proof}[Proof of Theorem \ref{theorem 1.1}(\textrm{ii})]
Let
\begin{equation*}
\left\{
\begin{aligned}
&0<v_1<\min\left\{v_0,\frac{\delta}{4}\right\};\\
 &\varepsilon_k:=2^{-k} \quad \textrm{for} \quad k=1,2,3\cdot\cdots;\\
&v_k:=\varepsilon_kv_{k-1}^2 \quad \textrm{for} \quad k=2,3,4\cdot\cdots\\
\end{aligned}
\right.
\end{equation*}
By induction, we have
$$
\begin{cases} 0<v_k<1\ \ \text{for}\ \  k=1,2,3\cdot\cdots;\\
 0<v_k\leq\varepsilon_k\ \ \text{for}\ \  k=1,2,3\cdot\cdots;\\
v_k\leq \varepsilon_kv_{k-1}\leq\frac{1}{2}v_{k-1}
\ \ \text{for}\ \  k=2,3,4\cdots;\\
\sum_{j=k+1}^\infty v_j\lesssim v_{k+1};\\
\sum_{j=1}^{k-1}\frac{1}{{v^4_j}}\lesssim \frac{1}{{v^4_{k-1}}}.
\end{cases}
$$
Upon defining
$$f:=\sum_{k=1}^\infty f_{v_k}.$$
and using Lemma \ref{lemma 2.1}, we obtain
$$\|f\|_{H^s(\mathbb{R})}\leq\sum_{k=1}^\infty\|f_{v_k}\|_{H^s(\mathbb{R})}\lesssim\sum_{k=1}^\infty{v_k}^{\frac{1}{2}-2s}
\leq\sum_{k=1}^\infty2^{-k(\frac{1}{2}-2s)}<\infty
\quad \forall\ \  0<s<\frac{1}{4}.$$
Also, via utilizing
$$\text{supp} f\subseteq\left(-\frac{\delta}{4},\frac{\delta}{4}\right)$$
and $t(x)$ in Lemma \ref{lemma 2.2}:
$$t(x)=\frac{x}{\Phi'\left(\frac{1}{v^2}\right)}=\frac{xv^2\sqrt{v^4+1}}{v^4+2},
$$
we achieve
\begin{align*}
\left|\mathcal{B}_{t_k(x)}f(x)\right|
&=\left|\sum_{j=1}^\infty \mathcal{B}_{t_k(x)}f_{v_j}(x)\right|\\
&\geq\left|\mathcal{B}_{t_k(x)}f_{v_k}(x)\right|-\left|\sum_{j=1}^{k-1} \mathcal{B}_{t_k(x)}f_{v_j}(x)\right|-\left|\sum_{j=k+1}^\infty \mathcal{B}_{t_k(x)}f_{v_j}(x)\right|\\
&\geq c_0-\left|\sum_{j=1}^{k-1} \mathcal{B}_{t_k(x)}f_{v_j}(x)\right|-\left|\sum_{j=k+1}^\infty \mathcal{B}_{t_k(x)}f_{v_j}(x)\right|,
\end{align*}
and then estimate the last two sums under $\frac{\delta}{2}<x<\delta$.
\begin{itemize}
	\item [$\rhd$] For $1\leq j\leq k-1$, by Lemma \ref{lemma 2.4}, we get
\begin{align*}
\left|\sum_{j=1}^{k-1} \mathcal{B}_{t_k(x)}f_{v_j}(x)\right|
&\lesssim\sum_{j=1}^{k-1} \frac{t_k(x)}{v_j^4}\\
&\thickapprox\frac{xv^2_k\sqrt{v^4_k+1}}{v^4_k+2}\frac{1}{v_{k-1}^4}\\
&\lesssim\frac{(\varepsilon_k v^2_{k-1})^2\sqrt{(\varepsilon_k v^2_{k-1})^4+1}}{(\varepsilon_k v^2_{k-1})^4+2}\frac{1}{v_{k-1}^4}\\
&\rightarrow0 \quad \textrm{as} \quad k\rightarrow\infty.
\end{align*}

	\item [$\rhd$] For $j\geq k+1$, by Lemma \ref{lemma 2.3}, we have
\begin{align*}
\left|\sum_{j=k+1}^{\infty} \mathcal{B}_{t_k(x)}f_{v_j}(x)\right|
&\lesssim\sum_{j=k+1}^{\infty} \frac{v_j}{t^{\frac{1}{2}}_k(x)}\\
&\lesssim\frac{1}{t^{\frac{1}{2}}_k(x)}v_{k+1}\\
&\lesssim\left(\frac{v_k^4+2}{v_k^2\sqrt{v_k^4+1}}\right)^{\frac{1}{2}}\varepsilon_{k+1}v_k^2\\
&\thickapprox\frac{v_k\sqrt{v_k^4+2}}{\sqrt[4]{v_k^4+1}}\varepsilon_{k+1}\\
&\rightarrow0 \quad \textrm{as} \quad k\rightarrow\infty.
\end{align*}
\end{itemize}
Accordingly, for $k\geq k_0$, we can get
$$|\mathcal{B}_{t_k(x)}f(x)|\geq\frac{c_0}{2} \quad  \textrm{for} \quad \frac{\delta}{2}<x<\delta.$$
\end{proof}

\section{Demonstration of Theorem \ref{theorem 1.10}}\label{s3}

\subsection{Proof of Theorem \ref{theorem 1.10}(i)}\label{s31}

This amounts to verifying  \eqref{maximal estimate}.

\begin{proof} [Proof of Theorem \ref{theorem 1.10}(i)]
 It is enough to deal with $\frac{1}{4}\leq s<\frac{1}{2}$ since the case $s=\frac{1}{2}$ follows as a consequence.
The $\alpha$-energy of $\mu$ is defined by
$$I_\alpha(\mu):=\int_{\mathbb{B}}\int_{\mathbb{B}}\frac{1}{|x-y|^\alpha}\,\textrm{d}\mu(x)\,\textrm{d}\mu(y).$$
From a dyadic decomposition it follows that if
$$\mu\in M^\alpha(\mathbb{B})~\ \ \&\ \  \alpha>1-2s$$
then
\begin{align*}
I_{1-2s}(\mu)
&=\int_{\mathbb{B}}\int_{\mathbb{B}}\frac{1}{|x-y|^{1-2s}}\,\textrm{d}\mu(x)\,\textrm{d}\mu(y)\\
&=\int_{\mathbb{B}}\sum_{j=0}^\infty\int_{B(y,2^{-j})\backslash B(y,2^{-j-1})}\frac{1}{|x-y|^{1-2s}}\,\textrm{d}\mu(x)\,\textrm{d}\mu(y)\\
&\leq\int_{\mathbb{B}}\sum_{j=0}^\infty\int_{B(y,2^{-j})}\frac{1}{2^{(-j-1)(1-2s)}}\,\textrm{d}\mu(x)\,\textrm{d}\mu(y)\\
&\leq\int_{\mathbb{B}}\sum_{j=0}^\infty2^{(j+1)(1-2s)}\mu(B(y,2^{-j}))\,\textrm{d}\mu(y)\\
&\leq\int_{\mathbb{B}}\sum_{j=0}^\infty2^{(j+1)(1-2s)}c_\alpha(\mu) 2^{-j\alpha}\,\textrm{d}\mu(y)\\
&\leq\int_{\mathbb{B}}c_\alpha(\mu) 2^{1-2s}\sum_{j=0}^\infty2^{j(1-2s-\alpha)}\,\textrm{d}\mu(y)\\
&\lesssim c_\alpha(\mu).
\end{align*}
Accordingly, it suffices to prove that
\begin{align}\label{4.1}
\left\|\sup_{(k,N)\in\mathbb N^2}\left|\mathcal{B}_{t_k}^Nf\right|\right\|_{L^1(d\mu)}\lesssim\sqrt{I_{1-2s}(\mu)}\|f\|_{\dot{H}^s(\mathbb{R})},
\end{align}
equivalently,
\begin{align}\label{4.2}
\left|\int_{\mathbb{B}}\mathcal{B}_{t(x)}^{N(x)}f(x)\omega(x)\,\textrm{d}\mu(x)\right|^2\lesssim I_{1-2s}(\mu)\|f\|^2_{\dot{H}^s(\mathbb{R})},
\end{align}
holds uniformly in the measurable functions
$$
\begin{cases} t(x): \mathbb{B}\rightarrow\mathbb{R};\\ N(x): \mathbb{B}\rightarrow[1,\infty);\\
\omega(x): \mathbb{B}\rightarrow\mathbb{S}^1=\{-1,1\}.
\end{cases}$$
Per applying Fubini's theorem and H\"{o}lder's inequality we can get
\begin{align*}
&\left|\int_{\mathbb{B}}\mathcal{B}_{t(x)}^{N(x)}f(x)\omega(x)\,\textrm{d}\mu(x)\right|^2\\
&\ \ \leq\int_{\mathbb{R}}|\hat{f}(\xi)|^2|\xi|^{2s}\,\textrm{d}\xi
\int_{\mathbb{R}}\left|\int_{\mathbb{B}}\psi\left(\frac{|\xi|}{N(x)}\right)e^{i(x\cdot\xi+t(x)|\xi|\sqrt{1+|\xi|^2})}\omega(x)\,\textrm{d}\mu(x)\right|^2
\frac{1}{|\xi|^{2s}}\,\textrm{d}\xi\\
&\ \ =\|f\|^2_{\dot{H}^s(\mathbb{R})}\int_{\mathbb{B}}\int_{\mathbb{B}}\int_{\mathbb{R}}\psi\left(\frac{|\xi|}{N(x)}\right)\psi\left(\frac{|\xi|}{N(y)}\right)
e^{i((x-y)\cdot\xi-(t(x)-t(y))|\xi|\sqrt{1+|\xi|^2})}\frac{1}{|\xi|^{2s}}\,\textrm{d}\xi\\
&\quad \times\omega(x)\omega(y)\,\textrm{d}\mu(x)\,\textrm{d}\mu(y).
\end{align*}
In order to prove  \eqref{4.2} it suffices to show that
\begin{align}\label{4.3}
&\int\int\int_{\mathbb{B}\times\mathbb{B}\times\mathbb{R}}\left(\cdot\cdot\cdot\right)\\
& \ \ :=\int_{\mathbb{B}}\int_{\mathbb{B}}\int_{\mathbb{R}}\psi\left(\frac{|\xi|}{N(x)}\right)\psi\left(\frac{|\xi|}{N(y)}\right)
e^{i((x-y)\cdot\xi-(t(x)-t(y))|\xi|\sqrt{1+|\xi|^2})}\frac{1}{|\xi|^{2s}}\,\textrm{d}\xi
\omega(x)\omega(y)\,\textrm{d}\mu(x)\,\textrm{d}\mu(y)\nonumber \\
&\ \ \lesssim I_{1-2s}(\mu)\nonumber
\end{align}
holds uniformly in the above-defined functions $t,N,\omega$.
By Lemma \ref{lemma 4.2}, we get
\begin{align*}
\int\int\int_{\mathbb{B}\times\mathbb{B}\times\mathbb{R}}\left(\cdot\cdot\cdot\right)& \lesssim\int_{\mathbb{B}}\int_{\mathbb{B}}\frac{1}{|x-y|^{1-2s}}\left|\omega(x)\omega(y)\right|\,\textrm{d}\mu(x)\,\textrm{d}\mu(y)\\
& \thickapprox\int_{\mathbb{B}}\int_{\mathbb{B}}\frac{1}{|x-y|^{1-2s}}\,\textrm{d}\mu(x)\,\textrm{d}\mu(y)\\
&\thickapprox I_{1-2s}(\mu),
\end{align*}
thereby reaching \eqref{4.3}.
\end{proof}

\subsection{Proof of Theorem \ref{theorem 1.10}(ii)}\label{s32}

This consists of two-side-inequalities.

\begin{proof}[Proof of Theorem \ref{theorem 1.10}(ii)]
From \eqref{maximal estimate} and the definition of $
\bar\alpha_{1,\mathcal{B}}(s)$, we have
$$\bar\alpha_{1,\mathcal{B}}(s)\leq1-2s.$$

Next we show
$$\bar\alpha_{1,\mathcal{B}}(s)\ge 1-2s.$$
To do so, consider
\begin{equation*}
\left\{
\begin{aligned}
 &\hat{f}=\chi_A;\\
&\,\textrm{d}\mu(x)=N^n\chi_E(x)\,\textrm{d}x;\\
&A=B(0,N);\\
&E=B(0,N^{-1}),\\
\end{aligned}
\right.
\end{equation*}
Clearly, we have
$$\mathcal{B}_t^Nf(x)=(2\pi)^{-n}\int_{B(0,N)}\psi\left(\frac{|\xi|}{N}\right)e^{i(x\cdot\xi+t|\xi|\sqrt{1+|\xi|^2})}\,\textrm{d}\xi.$$
\begin{itemize}
\item [$\rhd$] On the one hand, upon choosing $t=N^{-2}$, we see that the phase is close to zero for all $x\in E$, thereby deriving
\begin{align*}
\left\|\sup_{0<t<1}\left|\mathcal{B}_{t}^Nf(x)\right|\right\|_{L^1(\,\textrm{d}\mu)}
&=\left\|\sup_{0<t<1}\left|(2\pi)^{-n}\int_{B(0,N)}\psi\left(\frac{|\xi|}{N}\right)e^{i(x\cdot\xi+t|\xi|\sqrt{1+|\xi|^2})}\,\textrm{d}\xi\right|\right\|_{L^1(\,\textrm{d}\mu)}\\
&\gtrsim \left\|\int_{B(0,N))}\psi\left(\frac{|\xi|}{N}\right)\,\textrm{d}\xi\right\|_{L^1(\,\textrm{d}\mu)}\\
&\thickapprox\left\|\int_{B(0,N))}\psi\left(\frac{|\xi|}{N}\right)\,\textrm{d}\xi\right\|_{L^1(N^n\chi_E(x)\,\textrm{d}x)}\\
&\gtrsim N^n|A||E|\\
&\thickapprox N^n.
\end{align*}
\item[$\rhd$] On the other hand, we calculate
\begin{align*}
\sqrt{c_\alpha(\mu)}\|f\|_{H^s(\mathbb{R}^n)}
&=\sqrt{\sup_{(x,r)\in\mathbb R^n\times(0,\infty)}r^{-\alpha}{\mu\big(B(x,r)\big)}}\left(\int_{B(0,N))}\left(1+|\xi|^2\right)^s\,\textrm{d}\xi\right)^{\frac{1}{2}}\\
&\leq\sqrt{\sup_{(x,r)\in\mathbb R^n\times(0,\infty)}r^{-\alpha}{N^n|B(0,N^{-1})\cap B(x, r)|}}(1+N^2)^{\frac{s}{2}}N^{\frac{n}{2}}\\
&\lesssim\sqrt{\frac{N^n N^{-n}}{N^{-\alpha}}}N^{s+\frac{n}{2}}\\
&\lesssim N^{\frac{\alpha}{2}}N^{s+\frac{n}{2}}.
\end{align*}
\end{itemize}
Via letting $N\rightarrow\infty$, we get
$$\bar{\alpha}_{n,\mathcal{B}}(s)\geq n-2s,$$
thereby taking $n=1$ to reveal
$$\bar{\alpha}_{1,\mathcal{B}}(s)\geq 1-2s.$$
\end{proof}

\section{Demonstration of Theorem \ref{theorem 1.7}}\label{s4}

\subsection{Proof of Theorem \ref{theorem 1.7} (if)}\label{s41}
In order to prove the if-part, we recall two lemmas.
\begin{lemma}\label{lemma 2.5}
(\cite{SW} - Bessel's formula) If
$J_m(r)$ is the Bessel function defined by
$$J_m(r)=\frac{(\frac{r}{2})^m}{\Gamma(m+\frac{1}{2})\pi^{\frac{1}{2}}}\int_{-1}^1 e^{i rx}(1-x^2)^{m-\frac{1}{2}}\,\textrm{d}x\quad\text{as}\ \  m>-\frac{1}{2},
$$
then
$$J_m(r)=\sqrt{\frac{2}{\pi r}}\cos\left(r-\frac{\pi m}{2}-\frac{\pi}{4}\right)+O(r^{-\frac{3}{2}})\quad \text{as}\ \ r\rightarrow\infty.
$$
\end{lemma}
\begin{lemma}\label{lemma 2.6}
(\cite{M} - Pitt's inequality) If
\begin{equation*}
\left\{
\begin{aligned}
 &r\geq p;\\
&0\leq\alpha<\frac{1}{r};\\
&0\leq\alpha_1<1-\frac{1}{p};\\
&\alpha_1=\alpha+1-\frac{1}{r}-\frac{1}{p},\\
\end{aligned}
\right.
\end{equation*}
then
$$\left(\int_\mathbb{R}\left|\hat{f}(\xi)\right|^r|\xi|^{-\alpha r}\,\textrm{d}\xi\right)^{\frac{1}{r}}
\lesssim\left(\int_\mathbb{R}\left|f(x)\right|^p|x|^{\alpha_1 p}\,\textrm{d}x\right)^{\frac{1}{p}}.$$
\end{lemma}
\begin{proof}[Proof of Theorem \ref{theorem 1.7} (if)]

It suffices to prove the if-part with $s=\frac{1}{4}$ and $f$ being radial.

Firstly, \cite{SW} gives
$$\hat{f}(\xi)=(2\pi)^{\frac{n}{2}}|\xi|^{1-\frac{n}{2}}\int_0^\infty f(r) J_{\frac{n}{2}-1}(r|\xi|)r^{\frac{n}{2}}\,\textrm{d}r.
$$

Secondly, we set $t(x): \mathbb{R}^n\rightarrow\mathbb{R}$ be radial measurable and
$$\mathcal{B}_{t(x)}f(x):=(2\pi)^{-n}\int_{\mathbb{R}^n} e^{ix\cdot\xi+it(x)|\xi|\sqrt{1+|\xi|^2}}\hat{f}(\xi)\,\textrm{d}\xi, \quad x\in\mathbb{R}^n,
$$
whence $$\mathcal{B}_{t(u)}f(u)=(2\pi)^{\frac{n}{2}-n}u^{1-\frac{n}{2}}\int_0^\infty J_{\frac{n}{2}-1}(ru)e^{it(u)r\sqrt{1+r^2}}\hat{f}(r)r^{\frac{n}{2}}\,\textrm{d}r, \quad u>0,$$
where
$$
\begin{cases} \mathcal{B}_{t(u)}f(u)=\mathcal{B}_{t(x)}f(x)\ \ \text{as}\ \ u=|x|;\\
\hat{f}(r)=\hat{f}(\xi)\ \ \text{as}\ \ r=|\xi|.
\end{cases}
$$

Thirdly, in order to obtain \eqref{1.11}, it remains to prove
\begin{equation}\label{2.4}
\left(\int_0^\infty\left|\mathcal{B}_{t(u)}f(u)\right|^q u^{q(\frac{n}{2}-s)-1} \,\textrm{d}u\right)^{\frac{1}{q}}
\lesssim\left(\int_0^\infty\left|\hat{f}(r)\right|^2 r^{2s} r^{n-1} \,\textrm{d}r\right)^{\frac{1}{2}}.
\end{equation}
The following three steps will be carried out.

\begin{itemize}
\item[$\rhd$] Let
$$
\begin{cases} g(r):=\hat{f}(r)r^{s}r^{\frac{n}{2}-\frac{1}{2}}\quad\forall\quad r>0;\\
\left(\int_0^\infty\left|\hat{f}(r)\right|^2 r^{2s}r^{n-1} \,\textrm{d}r\right)^{\frac{1}{2}}
=\left(\int_0^\infty|g(r)|^2 \,\textrm{d}r\right)^{\frac{1}{2}}.
\end{cases}
$$
\item[$\rhd$] In order to control the left hand side of  \eqref{2.4}, we estimate
\begin{eqnarray*}
\mathcal{B}_{t(u)}f(u)u^{\frac{n}{2}-s-\frac{1}{q}}
&=&(2\pi)^{\frac{n}{2}-n}u^{1-\frac{n}{2}}u^{\frac{n}{2}-s-\frac{1}{q}}
\int_0^\infty J_{\frac{n}{2}-1}(ru)e^{it(u)r\sqrt{1+r^2}}\hat{f}(r)r^{\frac{n}{2}}\,\textrm{d}r\\
&=&(2\pi)^{-\frac{n}{2}}u^{1-s-\frac{1}{q}}\int_0^\infty J_{\frac{n}{2}-1}(ru)e^{it(u)r\sqrt{1+r^2}}r^{-s}r^{\frac{1}{2}}g(r)\,\textrm{d}r\\
&:=&(2\pi)^{-\frac{n}{2}}Dg(u),
\end{eqnarray*}
where $$Dg(u):=u^{1-s-\frac{1}{q}}\int_0^\infty J_{\frac{n}{2}-1}(ru)e^{it(u)r\sqrt{1+r^2}}r^{-s}r^{\frac{1}{2}}g(r)\,\textrm{d}r,\quad r>0.$$
and consequently,
$$\left(\int_0^\infty\left|\mathcal{B}_{t(u)}f(u)\right|^q u^{q(\frac{n}{2}-s)-1} \,\textrm{d}u\right)^{\frac{1}{q}}
=\left(\int_0^\infty\left|(2\pi)^{-\frac{n}{2}}Dg(u)\right|^q \,\textrm{d}u\right)^{\frac{1}{q}}.$$
It reduces to prove
\begin{equation*}\label{2.5}
\left(\int_0^\infty\left|Dg(u)\right|^q \,\textrm{d}u\right)^{\frac{1}{q}}
\lesssim\left(\int_0^\infty\left|g(r)\right|^2 \,\textrm{d}r\right)^{\frac{1}{2}}.
\end{equation*}
But, since the adjoint operator of $D$ is given by
$$D^\ast h(r)=r^{-s} r^{\frac{1}{2}}\int_0^\infty J_{\frac{n}{2}-1}(ru)e^{-it(u)r\sqrt{1+r^2}}u^{1-s-\frac{1}{q}}h(u)\,\textrm{d}u,\quad r>0,
$$
it suffices to prove
\begin{equation*}\label{2.6}
\|D^\ast h\|_{L^2(0,\infty)}\lesssim\|h\|_{L^p(0,\infty)},
\end{equation*}
where
\begin{equation*}
\left\{
\begin{aligned}
 &\frac{1}{p}+\frac{1}{q}=1;\\
&\frac{1}{4}\leq s<\frac{1}{2};\\
&2\leq q\leq\frac{2}{1-2s};\\
&1<\frac{2}{1+2s}\leq p\leq2.\\
\end{aligned}
\right.
\end{equation*}
Upon setting
$$\sigma:=\frac{1}{q}+s-\frac{1}{2},$$
we see
$$D^\ast h(r)=r^{-s}\int_0^\infty (ru)^{\frac{1}{2}}J_{\frac{n}{2}-1}(ru)e^{-it(u)r\sqrt{1+r^2}}u^{-\sigma}h(u)\,\textrm{d}u.$$
Therefore, we are led to estimate the Bessel function in the last formula.

\begin{itemize}

\item By Lemma \ref{lemma 2.5}, there exist $b_1$ and $b_2$ depending only on $n$ such that
\begin{align}\label{2.7}
t^\frac{1}{2}J_{\frac{n}{2}-1}(t)=b_1e^{it}+b_2e^{-it}+O\left(\min\left\{1,\frac{1}{t}\right\}\right),\quad t>0.
\end{align}
In fact, Lemma \ref{lemma 2.5} yields
$$J_{\frac{n}{2}-1}(t)=\sqrt{\frac{2}{\pi t}}\cos\left(t-\frac{\pi(n-1)}{4}\right)+O(t^{-\frac{3}{2}}) \quad \textrm{as} \quad t\rightarrow\infty,
$$
and so
\begin{eqnarray*}
t^\frac{1}{2}J_{\frac{n}{2}-1}(t)
&=&t^\frac{1}{2}\sqrt{\frac{2}{\pi t}}\cos\left(t-\frac{\pi(n-1)}{4}\right)+O(t^{-1})\\
&=&\sqrt{\frac{2}{\pi}}\cos\left(\frac{\pi(n-1)}{4}\right)\cos t+\sqrt{\frac{2}{\pi}}\sin\left(\frac{\pi(n-1)}{4}\right)\sin t+O(t^{-1})\\
&=&(b_1+b_2)\cos t -i(b_1-b_2)\sin t+O(t^{-1})\\
&=&b_1e^{it}+b_2e^{-it}+O(t^{-1}),
\end{eqnarray*}
where $$b_1=\frac{1}{2}\sqrt{\frac{2}{\pi }}\left(\cos\left(\frac{\pi(n-1)}{4}\right)-i\sin\left(\frac{\pi(n-1)}{4}\right)\right),$$
$$b_2=\frac{1}{2}\sqrt{\frac{2}{\pi }}\left(\cos\left(\frac{\pi(n-1)}{4}\right)+i\sin\left(\frac{\pi(n-1)}{4}\right)\right).$$

\item When $t>1$, we have
\begin{align}\label{2.8}
\left|t^\frac{1}{2}J_{\frac{n}{2}-1}(t)-(b_1e^{it}+b_2e^{-it})\right|\lesssim t^{-1}.
\end{align}

\item For the case of $0<t<1$ and $n\ge 2$, we have
$$
\begin{cases}
J_m(t)=\frac{(\frac{t}{2})^m}{\Gamma(m+\frac{1}{2})\pi^{\frac{1}{2}}}\int_{-1}^1 e^{i tx}(1-x^2)^{m-\frac{1}{2}}\,\textrm{d}x\quad\text{as}\ \  m>-\frac{1}{2};\\
|J_m(t)|\lesssim t^m \quad \text{as}\ \  m>-\frac{1}{2}~\ \&\ ~ t>0;\\
|J_m(t)|\lesssim t^{-\frac{1}{2}} \quad \text{as}\ \  m>-\frac{1}{2} ~ \ \&\  ~0<t\leq1;\\
|J_{\frac{n}{2}-1}(t)|\lesssim t^{-\frac{1}{2}} \quad \text{as}\ \ 0<t\leq1.
\end{cases}
$$
Accordingly, if $0<t\leq1$, then
\begin{eqnarray*}
\left|t^\frac{1}{2}J_{\frac{n}{2}-1}(t)-(b_1e^{it}+b_2e^{-it})\right|
&\leq&|t^\frac{1}{2}J_{\frac{n}{2}-1}(t)|+|b_1e^{it}|+|b_2e^{-it}|\\
&\lesssim&t^\frac{1}{2}t^{-\frac{1}{2}}+|b_1|+|b_2|\\
&\lesssim&1,
\end{eqnarray*}
and hence
\begin{align}\label{2.9}
\left|t^\frac{1}{2}J_{\frac{n}{2}-1}(t)-(b_1e^{it}+b_2e^{-it})\right|\lesssim 1.
\end{align}
\end{itemize}

\item[$\rhd$] Upon combining \eqref{2.8} with \eqref{2.9}, we get \eqref{2.7}. Now,
\eqref{2.7} derives
\begin{equation}\label{2.10}
\begin{cases}
\left|t^\frac{1}{2}J_{\frac{n}{2}-1}(t)-(b_1e^{it}+b_2e^{-it})\right|\lesssim\frac{1}{t},\quad t>1;\\
\left|t^\frac{1}{2}J_{\frac{n}{2}-1}(t)-(b_1e^{it}+b_2e^{-it})\right|\lesssim1,\quad 0<t\leq1,
\end{cases}
\end{equation}
where $b_1$ and $b_2$ depend only on $n$.
With the help of \eqref{2.10} we obtain
\begin{align}\label{2.12}
D^\ast h(r):=b_1B_1(r)+b_2B_2(r)+B_3(r),
\end{align}
where
\begin{equation*}
\left\{
\begin{aligned}
 &B_1(r)=r^{-s}\int_0^\infty e^{iru}e^{-it(u)r\sqrt{1+r^2}}u^{-\sigma}h(u)\,\textrm{d}u;\\
&B_2(r)=r^{-s}\int_0^\infty e^{-iru}e^{-it(u)r\sqrt{1+r^2}}u^{-\sigma}h(u)\,\textrm{d}u;\\
&|B_3(r)|\leq Cr^{-s}
\int_0^\infty\min\left\{1,\frac{1}{ru}\right\}u^{-\sigma}|h(u)|\,\textrm{d}u.\\
\end{aligned}
\right.
\end{equation*}
In what follows, we estimate $B_1(r)$, $B_2(r)$ and $B_3(r)$ respectively.
\begin{itemize}
\item Via defining $$B(r):=r^{-s}\int_0^\infty e^{iru}e^{-it(u)|r|\sqrt{1+|r|^2}}u^{-\sigma}h(u)\,\textrm{d}u,\quad r\in\mathbb{R},$$
and choosing a real-valued function $\rho\in C_c^\infty(\mathbb{R})$ such that
\begin{equation*}
\left\{
\begin{aligned}
 &\rho(\xi)=
\left\{
\begin{aligned}
 &1 \quad \textrm{as} \quad |\xi|\leq 1;\\
&0 \quad \textrm{as} \quad |\xi|\geq 2; \\
\end{aligned}
\right.\\
&\rho_N(\xi):=\rho\left(\frac{\xi}{N}\right) \quad \forall N\in \mathbb{N},\\
\end{aligned}
\right.
\end{equation*}
as well as letting
 $$B_N(r):=\rho_N(r)|r|^{-s}\int_0^\infty e^{iru}e^{-it(u)|r|\sqrt{1+|r|^2}}u^{-\sigma}h(u)\,\textrm{d}u,
 $$
we utilize Fubini's theorem to get
\begin{align}\label{2.14}
\int_\mathbb{R}|B_N(r)|^2\,\textrm{d}r
:=\int_0^\infty\int_0^\infty I(u, v)u^{-\sigma}h(u)v^{-\sigma}\overline{h(v)}\,\textrm{d}u\,\textrm{d}v,
\end{align}
where $$I(u, v):=\int_\mathbb{R} e^{i((u-v)r-(t(u)-t(v))|r|\sqrt{1+|r|^2})}|r|^{-s}\rho_N^2(r)\,\textrm{d}r.$$

\item By Lemma \ref{lemma 4.2}, we have
\begin{align}\label{2.15}
|I(u, v)|\lesssim\frac{1}{|u-v|^{\frac{1}{2}}}\ \ \text{under}\ \ s=\frac{1}{4}.
\end{align}
From \eqref{2.14}, \eqref{2.15} and Parseval's equality, we have
\begin{eqnarray*}
\|B_N\|_{L^2(\mathbb{R})}^2
&\lesssim& \int_0^\infty\int_0^\infty \frac{1}{|u-v|^{\frac{1}{2}}}u^{-\sigma}h(u)v^{-\sigma}|h(v)|\,\textrm{d}u\,\textrm{d}v\\
&\thickapprox& \int_\mathbb{R}\int_\mathbb{R}\frac{1}{|u-v|^{\frac{1}{2}}}u^{-\sigma}h_1(u)v^{-\sigma}|h_1(v)|\,\textrm{d}u\,\textrm{d}v\\
&\thickapprox& \int_\mathbb{R}I_{\frac{1}{2}}(v^{-\sigma}|h_1(v)|)(u)u^{-\sigma}|h_1(u)|\,\textrm{d}u\\
&\thickapprox& \int_\mathbb{R}|\xi|^{-\frac{1}{2}}(u^{-\sigma}|h_1(u)|)^\wedge(\xi)\overline{(u^{-\sigma}|h_1(u)|)^\wedge(\xi)}\,\textrm{d}\xi\\
&\thickapprox& \int_\mathbb{R}|\xi|^{-\frac{1}{2}}\left|(u^{-\sigma}|h_1(u)|)^\wedge(\xi)\right|^2\,\textrm{d}\xi,
\end{eqnarray*}
where
\begin{eqnarray*}
h_1(u)=
\begin{cases}
h(u) & \textrm{as} \quad u\geq0;
\cr 0 & \textrm{as} \quad u<0.
\end{cases}
\end{eqnarray*}

\item By Pitt's inequality, we have
$$\left(\int_\mathbb{R}\left|\hat{f}(\xi)\right|^2|\xi|^{-2s}\,\textrm{d}\xi\right)^{\frac{1}{2}}
\lesssim\left(\int_\mathbb{R}\left|f(x)\right|^p|x|^{(s+\frac{1}{2})p-1}\,\textrm{d}x\right)^{\frac{1}{p}},$$
where
\begin{equation*}
\left\{
\begin{aligned}
 &\frac{1}{4}\leq s<\frac{1}{2};\\
&\frac{2}{1+2s}\leq p\leq2.\\
\end{aligned}
\right.
\end{equation*}
Consequently,
\begin{eqnarray*}
\|B_N\|_{L^2(\mathbb{R})}^2
&\lesssim& \left(\int_\mathbb{R}\left|u^{-\sigma}h_1(u)\right|^p|u|^{\frac{3}{4}p-1}\,\textrm{d}u\right)^{\frac{2}{p}}\\
&\thickapprox& \left(\int_\mathbb{R}u^{-\sigma p+\frac{3}{4}p-1}|h_1(u)|^p \,\textrm{d}u\right)^{\frac{2}{p}}\\
&\thickapprox& \|h_1\|_{L^p(\mathbb{R})}^2\\
&\thickapprox& \|h\|_{L^p(0,\infty)}^2,
\end{eqnarray*}
namely, $$\|B_N\|_{L^2(\mathbb{R})}\lesssim\|h\|_{L^p(0,\infty)}.$$
According to Fatou's lemma we obtain
$$\left(\int_\mathbb{R}|B(r)|^2\,\textrm{d}r\right)^{\frac{1}{2}}\lesssim\|h\|_{L^p(0,\infty)}.$$
Furthermore,
\begin{eqnarray*}
\left(\int_0^\infty|B_i(r)|^2\,\textrm{d}r\right)^{\frac{1}{2}}
\lesssim \left(\int_\mathbb{R}|B(r)|^2\,\textrm{d}r\right)^{\frac{1}{2}}
\lesssim \|h\|_{L^p(0,\infty)}\ \ \forall\ \ i=1,2.
\end{eqnarray*}
\item Next we dominate $B_3(r)$ according to two situations.
\begin{enumerate}
  \item[\rm(i)] Under $0<r<1$, it suffices to prove
  $$\left(\int_0^1 |B_3(r)|^2\,\textrm{d}r\right)^{\frac{1}{2}}\lesssim\|h\|_{L^p(0,\infty)}.$$
  H\"{o}lder's inequality derives
  \begin{eqnarray*}
  |B_3(r)|
  &\lesssim& \int_0^\infty\min\left\{1,\frac{1}{ru}\right\}u^{-\sigma}|h(u)|\,\textrm{d}u\\
  &\lesssim& \int_0^{\frac{1}{r}}u^{-\sigma}|h(u)|\,\textrm{d}u+\frac{1}{r}\int_\frac{1}{r}^\infty u^{-1-\sigma}|h(u)|\,\textrm{d}u\\
  &\lesssim& \left(\int_0^{\frac{1}{r}}u^{-\sigma q}\,\textrm{d}u\right)^\frac{1}{q}\|h\|_{L^p(0,\infty)}
  +\frac{1}{r}\left(\int_\frac{1}{r}^\infty u^{(-1-\sigma)q}\,\textrm{d}u\right)^\frac{1}{q} \|h\|_{L^p(0,\infty)}\\
  &\thickapprox& r^{\sigma-\frac{1}{q}}\|h\|_{L^p(0,\infty)}\\
  &\thickapprox& r^{s-\frac{1}{2}}\|h\|_{L^p(0,\infty)},
  \end{eqnarray*}
 whence
  $$\left(\int_0^1 |B_3(r)|^2\,\textrm{d}r\right)^{\frac{1}{2}}
  \lesssim\left(\int_0^1 r^{2s-1}\,\textrm{d}r\right)^{\frac{1}{2}}\|h\|_{L^p(0,\infty)}
  \lesssim\|h\|_{L^p(0,\infty)}.$$

\item[\rm(ii)] Under $r\geq1$, it suffices to prove
  $$\left(\int_1^\infty |B_3(r)|^2\,\textrm{d}r\right)^{\frac{1}{2}}\lesssim\|h\|_{L^p(0,\infty)}.$$
  On the one hand, since
  \begin{eqnarray*}
  |B_3(r)|
  &\lesssim&  r^{-s}\int_0^{\frac{1}{r}}u^{-\sigma}|h(u)|\,\textrm{d}u+ r^{-s}\int_\frac{1}{r}^\infty \frac{1}{ru}u^{-\sigma}|h(u)|\,\textrm{d}u\\
  &\thickapprox&  r^{-s}\int_0^{\frac{1}{r}}u^{-\sigma}|h(u)|\,\textrm{d}u+r^{-1-s}\int_\frac{1}{r}^\infty u^{-\sigma-1}|h(u)|\,\textrm{d}u\\
  &:=&E_1(r)+E_2(r),
  \end{eqnarray*}
  setting
  $$F_1(t):=\frac{1}{t}E_1\left(\frac{1}{t}\right),\quad 0<t<1$$
   yields
  \begin{eqnarray*}
  F_1(t)
  &=&  t^{-1+s}\int_0^t u^{-\sigma}|h(u)|\,\textrm{d}u\\
  &\leq& \int_0^t (t-u)^{-1+s} u^{-\sigma}|h(u)|\,\textrm{d}u\\
  &\leq& \int_\mathbb{R}|t-u|^{-1+s} u^{-\sigma}|h_1(u)|\,\textrm{d}u\\
  &\leq&I_s(u^{-\sigma}|h_1(u)|)(t),
  \end{eqnarray*}
  Similarly to the estimate for $\|B_N\|_{L^2(\mathbb{R})}$, we have
\begin{eqnarray*}
  \left(\int_1^\infty \left|E_1(r)\right|^2\,\textrm{d}r\right)^{\frac{1}{2}}
  &\lesssim& \left(\int_0^1 |F_1(t)|^2\,\textrm{d}t\right)^{\frac{1}{2}}\\
  &\lesssim&\left(\int_\mathbb{R} \left|I_s(u^{-\sigma}|h_1(u)|)(t)\right|^2\,\textrm{d}t\right)^{\frac{1}{2}}\\
  &\thickapprox&\left(\int_\mathbb{R} \left|(u^{-\sigma}|h_1(u)|)^\wedge(\xi)\right|^2 |\xi|^{-2s}\,\textrm{d}\xi\right)^{\frac{1}{2}}\\
  &\lesssim&\|h_1\|_{L^p(\mathbb{R})}\\
  &\thickapprox&\|h\|_{L^p(0,\infty)},
  \end{eqnarray*}
thereby producing $$\left(\int_1^\infty |E_1(r)|^2\,\textrm{d}r\right)^{\frac{1}{2}}
\lesssim\|h\|_{L^p(0,\infty)}.$$

On the other hand, we want to prove $$\left(\int_1^\infty |E_2(r)|^2\,\textrm{d}r\right)^{\frac{1}{2}} \lesssim\|h\|_{L^p(0,\infty)}.$$
Now, setting
$$F_2(t):=\frac{1}{t}E_2\left(\frac{1}{t}\right),\quad 0<t<1
$$
implies
  \begin{eqnarray*}
  F_2(t)
  &=&  t^{s}\int_t^\infty u^{-1-\sigma}|h(u)|\,\textrm{d}u\\
  &\leq& \int_t^\infty u^s u^{-1-\sigma} |h(u)|\,\textrm{d}u\\
  &\leq& \int_\mathbb{R}|t-u|^{-1+s} u^{-\sigma}|h_1(u)|\,\textrm{d}u\\
  &\leq&I_s(u^{-\sigma}|h_1(u)|)(t).
  \end{eqnarray*}
  Similarly to the estimate of
  $$\left(\int_1^\infty |E_1(r)|^2\,\textrm{d}r\right)^{\frac{1}{2}} ,$$ we have
  $$\left(\int_1^\infty |E_2(r)|^2\,\textrm{d}r\right)^{\frac{1}{2}}
  \lesssim\|h\|_{L^p(0,\infty)}.$$
  Combining \eqref{2.12} with the estimates for $B_1(r)$-$B_2(r)$-$B_3(r)$ yields \eqref{2.4} and then \eqref{1.11}.
\end{enumerate}
\end{itemize}
\end{itemize}

\end{proof}
\subsection{Proof of Theorem \ref{theorem 1.7} (only-if)}

This part is constructive.

\begin{proof}[Proof of Theorem \ref{theorem 1.7} (only-if)]
To do so, choose a nonnegative radial function $\varphi\in C_c^\infty(\mathbb{R}^n)$ such that
$$
\begin{cases}
\text{supp} \varphi \subset\{\xi:1\leq|\xi|\leq2\};\\
\varphi(\xi)=1\ \ \text{as}\ \ \frac{5}{4}\leq|\xi|\leq \frac{7}{4};\\
\hat{f}(\xi):=\varphi\left(\frac{\xi}{\lambda}\right)\ \ \text{as}\ \ \lambda>0.
\end{cases}
$$
It is easy to see that
$$\|f\|_{\dot{H}^s(\mathbb{R}^n)}\approx\lambda^{\frac{n}{2}+s}.$$
Since
\begin{eqnarray*}
\mathcal{B}_tf(x)
&=&\int_{\mathbb{R}^n} e^{ix\cdot\xi}e^{it|\xi|\sqrt{1+|\xi|^2}}\varphi\left(\frac{\xi}{\lambda}\right)\,\textrm{d}\xi\\
&=&\lambda^n \int_{\mathbb{R}^n} e^{i\lambda x\cdot\eta}e^{it|\lambda\eta|\sqrt{1+|\lambda\eta|^2}}\varphi(\eta)\,\textrm{d}\eta,
\end{eqnarray*}
taking $t=0$ derives
$$\mathcal{B}_{0}f(x)=\lambda^n \int_{\mathbb{R}^n} e^{i\lambda x\cdot\eta}\varphi(\eta)\,\textrm{d}\eta=\lambda^n\hat{\varphi}(\lambda x).$$
Also, since
$$\hat{\varphi}(0)=\int_{\mathbb{R}^n} \varphi(x)\,\textrm{d}x\geq\int_{\frac{5}{4}\leq|\xi|\leq \frac{7}{4}} \varphi(x)\,\textrm{d}x>1,$$
there exist $0<\delta<\frac{\lambda}{2}$ such that
$$\hat{\varphi}(\lambda x)>\frac{1}{2} \quad\forall\quad |x|<\frac{\delta}{\lambda}.$$
Accordingly,
$$|\mathcal{B}^{\ast\ast}f(x)|\geq|\mathcal{B}_{0}f(x)|\geq c_0\lambda^n\ \ \&\ \
c_0=\frac{1}{2(2\pi)^n}.
$$

Now, if \eqref{1.11} holds, then
\begin{align*}
 \lambda^{\frac{n}{2}+s}
&\approx\|f\|_{\dot{H}^s(\mathbb{R}^n)}\\
&\gtrsim \left(\int_{\mathbb{R}^n}|\mathcal{B}^{\ast\ast}f(x)|^q |x|^\alpha \,\textrm{d}x\right)^\frac{1}{q}\\
&\geq\left(\int_{B(0,1)}|\mathcal{B}^{\ast\ast}f(x)|^q |x|^\alpha \,\textrm{d}x\right)^\frac{1}{q}\\
&\geq c_0\left(\int_{|x|<\frac{\delta}{\lambda}}\lambda^{nq} |x|^\alpha \,\textrm{d}x\right)^\frac{1}{q}\\
&=c_1 \lambda^{n-\frac{\alpha+n}{q}},
\end{align*}
where $$c_1=c_0\left(\frac{\omega_{n-1}\delta^{\alpha+n}}{\alpha+n}\right)^\frac{1}{q}$$ and $\omega_{n-1}$ is the area of unit sphere in $\mathbb{R}^n$.
Therefore,
\begin{equation}
\label{eFe}
\lambda^{n-\frac{\alpha+n}{q}}\lesssim\lambda^{\frac{n}{2}+s}.
\end{equation}
\begin{itemize}
\item[$\rhd$] Upon letting $\lambda\rightarrow\infty$ in \eqref{eFe}, we get
$$\alpha \geq q\Big(\frac{n}{2}-s\Big)-n.$$
\item[$\rhd$] Upon letting $\lambda\rightarrow0$ in \eqref{eFe}, we get
$$\alpha \leq q\Big(\frac{n}{2}-s\Big)-n.$$
\end{itemize}
As a consequence, we have
$$\alpha =q\Big(\frac{n}{2}-s\Big)-n.$$

\end{proof}



\bigskip

\noindent  Dan Li

\smallskip

\noindent  Laboratory of Mathematics and Complex Systems
(Ministry of Education of China),
School of Mathematical Sciences, Beijing Normal University,
Beijing 100875, People's Republic of China

\smallskip

\noindent {\it E-mails}: \texttt{danli@mail.bnu.edu.cn}

\bigskip

\noindent Junfeng Li(Corresponding author)

\smallskip

\noindent School of Mathematical Sciences, Dalian University of Technology, Dalian, LN, 116024, China

\smallskip

\noindent{\it E-mail}: \texttt{junfengli@dlut.edu.cn}

\bigskip

\noindent Jie Xiao

\smallskip

\noindent Department of Mathematics and Statistics,
		Memorial University, St. John's, NL A1C 5S7, Canada

\smallskip

\noindent{\it E-mail}: \texttt{jxiao@math.mun.ca}

\end{document}